\documentclass[10pt,a4paper]{amsart}

\usepackage{amsmath}
\usepackage{amsthm}
\usepackage{enumerate} 
\usepackage{amssymb}
\usepackage{graphicx}
\usepackage[all]{xy}
\usepackage{hyperref}
\usepackage{aliascnt}
\usepackage{tikz}
\usetikzlibrary{decorations.pathreplacing,calligraphy}
\usepackage{stmaryrd} 
\usepackage{mathtools}
\usepackage{verbatim} 
\usepackage{txfonts} 
\usepackage{lmodern}
\usepackage{setspace}   

\usepackage[hmargin=2.9cm]{geometry}
\newtheorem{lma}{Lemma}[section]

\newaliascnt{thmCt}{lma}
\newtheorem{thm}[thmCt]{Theorem}
\aliascntresetthe{thmCt}

\newaliascnt{corCt}{lma}
\newtheorem{cor}[corCt]{Corollary}
\aliascntresetthe{corCt}

\newaliascnt{propCt}{lma}
\newtheorem{prop}[propCt]{Proposition}
\aliascntresetthe{propCt}

\newtheorem*{thm*}{Theorem}
\newtheorem*{dfn*}{Definition}
\newtheorem*{cor*}{Corollary}
\newtheorem*{rmk*}{Remark}
\newtheorem*{prop*}{Proposition}
\newtheorem*{exa*}{Example}

\theoremstyle{definition}

\newaliascnt{prgCt}{lma}
\newtheorem{prg}[prgCt]{}
\aliascntresetthe{prgCt}

\newaliascnt{dfnCt}{lma}
\newtheorem{dfn}[dfnCt]{Definition}
\aliascntresetthe{dfnCt}

\newaliascnt{rmkCt}{lma}
\newtheorem{rmk}[rmkCt]{Remark}
\aliascntresetthe{rmkCt}

\newaliascnt{rmksCt}{lma}

\aliascntresetthe{rmksCt}

\newaliascnt{ntnCt}{lma}
\newtheorem{ntn}[ntnCt]{Notation}
\aliascntresetthe{ntnCt}

\newaliascnt{qstCt}{lma}

\aliascntresetthe{qstCt}

\newaliascnt{prblCt}{lma}

\aliascntresetthe{prblCt}

\newaliascnt{exaCt}{lma}

\aliascntresetthe{exaCt}

\newcommand{\CC}{\mathcal{C}}

\newcommand{\T}{\mathbb{T}}
\newcommand{\C}{\mathbb{C}}
\newcommand{\N}{\mathbb{N}}

\newcommand{\R}{\mathbb{R}}
\newcommand{\Z}{\mathbb{Z}}
\newcommand{\K}{\mathrm{K}}
\newcommand{\KK}{\mathcal{K}}
\newcommand{\ZZ}{\mathcal{Z}}

\DeclareMathOperator{\im}{im}

\DeclareMathOperator{\Int}{int}
\DeclareMathOperator{\diag}{diag}
\DeclareMathOperator{\card}{card}
\DeclareMathOperator{\dcp}{dcp_{ch}}
\DeclareMathOperator{\diam}{diam}

\DeclareMathOperator{\spectrum}{sp}

\DeclareMathOperator{\cokernel}{coker}

\newcommand{\alg}{\mathrm{alg}}

\newcommand{\CatCa}{C^*}

\DeclareMathOperator{\Tr}{Tr}
\DeclareMathOperator{\wsp}{wsp}

\DeclareMathOperator{\Cu}{Cu}

\DeclareMathOperator{\NCCW}{NCCW}
\DeclareMathOperator{\CW}{CW}
\DeclareMathOperator{\AI}{AI}

\DeclareMathOperator{\A}{A}

\DeclareMathOperator{\ASH}{ASH}
\DeclareMathOperator{\AH}{AH}

\DeclareMathOperator{\AF}{AF}
\DeclareMathOperator{\Lsc}{Lsc}

\DeclareMathOperator{\supp}{supp}
\DeclareMathOperator{\Hom}{Hom}
\DeclareMathOperator{\id}{id}

\DeclareMathAlphabet{\mymathbb}{U}{bbold}{m}{n}

\begin{document}
\onehalfspacing
\title{Towards a classification of unitary elements of $\CatCa$-algebras}
\author{Laurent Cantier}

\address{Laurent Cantier\newline
Departament de Matem\`{a}tiques \\
Universitat Aut\`{o}noma de Barcelona \\
08193 Bellaterra, Spain\newline
Institute of Mathematics\\ Czech Academy of Sciences\\ Zitna 25\\ 115 67 Praha 1\\ Czechia}
\email[]{laurent.cantier@uab.cat}

\thanks{The author was supported by the Spanish Ministry of Universities and the European Union-NextGenerationEU through a Margarita Salas grant and partially supported by the project PID2020-113047GB-I00.}

\begin{abstract}
In \cite{C22} the author conjectures and partially shows that the Cuntz semigroup classifies unitary elements of unital $\AF$-algebras. We provide a complete proof by addressing the existence part of the conjecture, under a mild adjustment of both domain and codomain of the functor $\Cu$. We also tackle the classification beyond the $\AF$ case and more particularly, we look at unitary elements of what we call $\AH_1$-algebras. We obtain positive progress as far as the existence part is concerned. Nevertheless, we reveal that extra information is needed for the uniqueness part of the classification that the Cuntz semigroup fails to capture. 
\end{abstract}
\maketitle

\section{Introduction}

Classification of $\CatCa$-algebras is a vast and long-standing field in operator algebras. A powerful and modern approach, which consists of classifying *-homomorphisms has yielded many deep classification results, e.g. \cite{R12}, \cite{EGLN21}, or \cite{GLN1,GLN2}. One of the most notable results ought to be the completion of the Elliott's conjecture where the work of many people has yielded the classification of a large class within the simple $\mathcal{Z}$-stable setting in terms of $\K$-theory and traces. See \cite{W18} for a general overview. 

In the non-simple setting, the most elegant result might very well be the classification of unital one-dimensional $\NCCW$-complexes with trivial $\K_1$-group and their inductive limits by means of the Cuntz semigroup obtained by Robert in \cite{R12}. It is worth mentioning that the main proof is based on the previous work done by Ciuperca and Elliott in \cite{CE08} followed by the work of Robert and Santiago in \cite{RS10} where they engage the classification of *-homomorphisms from $\CC_0(]0,1])$ to any $\CatCa$-algebra of stable rank one in terms of Cuntz semigroup. Uniqueness part of the classification is done in \cite{CE08} and improved in \cite{RS10} while the existence part of the classification is achieved in \cite{RS10}. Subsequently, Robert was able to expand the domain $\CC_0(]0,1])$ to the class of direct sums one dimensional $\NCCW$-complexes with trivial $\K_1$-group and their inductive limits. As an immediate corollary following from the well-known approximate intertwining argument, a complete classification of the latter class of $\CatCa$-algebras (including $\AI$-algebras) is obtained. 

This naturally raises the questions: can we go beyond the trivial $\K_1$-group hypothesis? In particular, could we apply a similar procedure to classify e.g. all $\A\!\T$-algebras?

A first step towards this direction would consist in classifying *-homomorphisms from $\CC(\T)$ to distinct classes of $\CatCa$-algebras. A main obstruction arises if our candidate is again the Cuntz semigroup alone and our codomain class is somewhat too large. This has been illustrated in many situations (e.g. in \cite{C23}) and some kind of $\K_1$-information or its Hausdorffized version will be needed. Note that these questions have also been raised in \cite[Question 16.7]{GP23}.

This paper is meant to serve as a starting point for this new and promising investigation. The results we obtain are multiple. We first focus on a well-chosen codomain class, namely the unital $\AF$-algebras. We follow up the work done in \cite{C22} and we complete the classification of unitary elements of these $\CatCa$-algebras by means of the Cuntz semigroup. In other words, we show that the Cuntz semigroup classifies *-homomorphisms from $\CC(\T)$ to any unital $\AF$-algebra $A$. To be more precise, we use a mild adjustment of the functor $\Cu\colon \mathcal{T}\!\star\AF\longrightarrow \Cu_{\wsp} $ , that we now define from a join subcategory containing all unital *-homomorphisms from $\mathcal{C}(\T)$ to any $\AF$-algebra, to the wide subcategory of $\Cu$-semigroups with what we call \emph{weakly semiprojective $\Cu$-morphisms}. 

\begin{thm*}
Let $A$ be a unital $\AF$-algebra. Let $\alpha:\Lsc(\T,\overline{\N})\longrightarrow \Cu(A)$ be a weakly semiprojective $\Cu$-morphism with $\alpha(1_\T)=[1_A]$. Then there exists a unitary element $u$ in $A$, unique up to approximate unitary equivalence, such that $ \alpha= \Cu(\varphi_u)$. 

As a consequence, $\Cu\colon \mathcal{T}\!\star\AF\longrightarrow \Cu_{\wsp} $  classifies unitary elements of unital $\AF$-algebras.
\end{thm*}

Note that throughout the paper, we may interchangeably speak about unitary elements of a $\CatCa$-algebra $A$ and *-homomorphisms from $\CC(\T)$ to $A$, due to the bijective correspondence between unitary elements of $A$ and $\Hom_{\CatCa}(\CC(\T),A)$. (See more details below.) 

Subsequently, we naturally go beyond the $\AF$ case and study larger classes of $\CatCa$-algebras.  
Regarding uniqueness, we are able to obtain positive results in the real rank zero setting. See \autoref{cor:RR0}. Nevertheless, the key point is that uniqueness fails in general. We exhibit a couple of examples providing obstructions, even in the simple, real rank one, trivial $\K_1$-group setting. 

These examples suggest that the de la Harpe-Skandalis determinant, closely related to the Hausdorffized algebraic $\K_1$-group, should be added to the Cuntz semigroup to obtain satisfactory results. 
We expose both examples in the following theorem, but the reader should be aware that they will appear in separate instances since their nature and purpose in the exposition of the manuscript are quite different. 

\begin{thm*} Let $A$ be either the Jiang-Su algebra $\mathcal{Z}$ or $\CC(X)\otimes M_{2^\infty}$ where $X$ is any one-dimensional compact $\CW$-complex. Then there exists two unitary elements $u,v$ in $A$, i.e. two *-homomorphisms $\varphi_u,\varphi_v:\CC(\T)\longrightarrow A$, such that $\Cu(\varphi_u)=\Cu(\varphi_v)$ and yet $\varphi_u\nsim_{aue} \varphi_v$.
\end{thm*}

With regard to existence, we obtain very satisfactory results. Firstly, we prove a technical result (\autoref{thm:existencefx}) for any $\CatCa$-algebra $B$ given by \textquoteleft building blocks\textquoteright\ over 1-dimensional compact $\CW$-complexes. More concretely, we construct an approximate lift, up to arbitrary precision, for any abstract $\Cu$-morphism $\alpha\colon\Lsc(\T,\overline{\N})\longrightarrow \Cu(B)$, which yields the following result by means of what we call \emph{Cauchy sequences of $\Cu$-morphisms}.

\begin{thm*}
Let $B$ be a direct sum of matrix algebras of continuous functions over 1-dimensional compact $\CW$-complexes. Let $\alpha:\Lsc(\T,\overline{\N})\longrightarrow \Cu(B)$ be a $\Cu$-morphism such that $\alpha(1_\T)=[1_B]$.

Then there exists a sequence of unitary elements $(u_n)_n$ in $B$ such that $(\Cu(\varphi_{u_n}))_n$ is Cauchy and converges towards $\alpha$ with respect to $dd_{\Cu}$.
\end{thm*}

Let us mention that with current methods, the above does not suffice to conclude to an actual lift, as it would also require a uniqueness result. 
Thereafter, we are still able to deduce the following, which too, can only be stated in terms of Cauchy sequences.

\begin{thm*}
Let $A$ be a unital $\AH_1$-algebra. Let $\alpha:\Lsc(\T,\overline{\N})\longrightarrow \Cu(A)$ be a weakly semiprojective $\Cu$-morphism with $\alpha(1_\T)=[1_A]$.

Then there exists a sequence of unitary elements $(u_n)_n$ in $A$ such that $(\Cu(\varphi_{u_n}))_n$ is Cauchy and converges towards $\alpha$ with respect to $dd_{\Cu}$.
\end{thm*}

We end the manuscript with open questions and partial leads towards future classification results.
\textbf{Acknowledgments.} 
The author would like to thank L. Robert for many fruitful discussions and the conception of both of the examples providing obstructions outside the $\AF$ setting, and also the proofreader for comments that have improved the redaction of the manuscript. 

\section{Preliminaries}
$\hspace{-0,34cm}\bullet\,\,\textbf{Unitary elements - Classification of *-homomorphisms}.$ Let us first establish some details about the classification of *-homomorphisms from $\CC(\T)$. There is a bijective correspondence between the set $\mathcal{U}(A)$ of unitary elements of a (unital) $\CatCa$-algebra $A$ and the set $\Hom_{\CatCa,1}(\CC(\T),A)$ of unital *-homomorphisms from $\CC(\T)$ to $A$, defined through functional calculus. More specifically, we have 
\[
\begin{array}{ll}
	\varphi:\mathcal{U}(A)\simeq \Hom_{\CatCa,1}(\mathcal{C}(\T),A)\\
	\hspace{1cm}u\longmapsto \varphi_u
\end{array}
\] 
where $\varphi_u(f):=f(u)$ for any $f\in \CC(\T)$. 

Let $A,B$ be unital $\CatCa$-algebras and let $\phi,\psi:A\longrightarrow B$ be *-homomorphisms. We say that $\phi$ is \emph{approximately unitarily equivalent to $\psi$} and we write $\phi\sim_{aue} \psi$ if, there exists a sequence $(w_n)_n$ of unitary elements in $B$ such that $w_n\phi(x)w_n^*\underset{n\rightarrow\infty}\longrightarrow\psi(x)$ for all $x\in A$. (Similarly, we say that two unitary elements $u,v$ of a unital $\CatCa$-algebra $A$ are approximately unitarily equivalent if, there exists a sequence $(w_n)_n$ of unitary elements in $A$ such that $w_n u w_n^*\underset{n\rightarrow\infty}\longrightarrow v$.)

Let $F:\CatCa\longrightarrow\CC$ be a functor. We say that $F$ \emph{classifies *-homomorphisms from $A$ to $B$} if $\Hom_{\CC}(F(A),F(B))\simeq \Hom_{\CatCa}(A,B)/\sim_{aue}$. The injectivity of the latter map is often referred to as the \emph{uniqueness} part of the classification, while the surjectivity is often referred to as the \emph{existence} part of the classification. 

In our context, we shall say that a functor $F$ \emph{classifies unitary elements of $A$} whenever $F$ classifies unital *-homomorphisms from $\CC(\T)$ to $A$. We may naturally say that $F$ \emph{distinguishes unitary elements of $A$} whenever we (only) have the uniqueness part of the classification.\\

$\hspace{-0,34cm}\bullet\,\,\textbf{Cuntz semigroups - Distance between morphisms}.$ 
The Cuntz semigroup was introduced in \cite{C78} and has been a powerful tool for classification since the milestone paper \cite{CEI08}. We refer the reader to \cite{GP23} for a state of the art survey on the matter where they shall find all the basics they might need. (The reader can also check the preliminaries of  \cite{C23}.) Another fundamental and complete paper can be found in \cite{APT18}. Let us briefly remind the basic definitions and a specific property regarding the category $\Cu$.
\begin{prg}
Let $(S,\leq)$ be an ordered monoid and let $x,y$ in $S$. We say that $x$ is \emph{way-below} $y$ and we write $x\ll y$ if, for all increasing sequences $(z_n)_{n\in\N}$ in $S$ that have a supremum, if $\sup\limits_{n\in\N} z_n\geq y$, then there exists $k$ such that $z_k\geq x$. This is an auxiliary relation on $S$ called the \emph{way-below relation} or the \emph{compact-containment relation}. In particular, $x\ll y$ implies $x\leq y$, and we say that $x$ is a \emph{compact element} whenever $x\ll x$. 

We say that $S$ is an \emph{abstract Cuntz semigroup}, or a $\Cu$-semigroup, if $S$ satisfies the above axioms.

(O1) Every increasing sequence of elements in $S$ has a supremum. 

(O2) For any $x\in S$, there exists a $\ll$-increasing sequence $(x_n)_{n\in\N}$ in $S$ such that $\sup\limits_{n\in\N} x_n= x$.

(O3) Addition and the compact containment relation are compatible.

(O4) Addition and suprema of increasing sequences are compatible.

We say that a map $\alpha:S\longrightarrow T$ is a $\Cu$-morphism if $\alpha$ is an ordered monoid morphism preserving the compact-containment relation and suprema of increasing sequences.
\end{prg}

We recall some facts on the comparison of $\Cu$-morphisms and more specifically on $\Cu$-morphisms with domain $\Lsc(\T,\overline{\N})\simeq \Cu(\CC(\T))$. A metric on $\Hom_{\Cu}(\Lsc(\T,\overline{\N}),T)$ has been introduced in several papers such as \cite[Definition 3.1]{JSV18}.
\begin{prg}[$\Cu$-metric]
Let $\alpha, \beta:\Lsc(\T,\overline{\N})\longrightarrow T$ be any two $\Cu$-morphisms. We define
\[  d_{\Cu}(\alpha,\beta):=\inf \Bigl\{ r>0\mid \forall U\in\mathcal{O}(\T), \alpha(1_{U})\leq\beta(1_{U_{r}})  \text{ and }  \beta(1_{U})\leq\alpha(1_{U_{r}}) \Bigr\}
\] 
where $\mathcal{O}(\T):=\{$Open sets of $\T\}$ and $U_r:=\underset{x\in U}{\cup}B_r(x)$. 
If the infimum does not exist, we set the value to $\infty$.
This defines a metric that we refer to as \emph{the $\Cu$-metric}.
\end{prg}

On the other hand, a topologically equivalent semi-metric has been introduced in \cite{C22}. To define it, we first need to introduce the abstract notion of comparison of $\Cu$-morphisms.

\begin{dfn}\cite[Definition 3.9]{C22}
Let $S,T$ be $\Cu$-semigroups and let $\Lambda\subseteq S$. Consider maps $\alpha,\beta:\Lambda\longrightarrow T$ preserving the order and the compact-containment relation.

We say that $\alpha$ and $\beta$ \emph{compare on $\Lambda$} and we write $\alpha\underset{\Lambda}{\simeq}\beta $ if, for any $g,h\in\Lambda$ such that $g\ll h$ (in $S$), we have that $\alpha(g)\leq \beta(h)$ and $\beta(g)\leq \alpha(h)$.
\end{dfn}

\begin{prg}[Discrete $\Cu$-semimetric]\label{prg:discretedistance}
Let $n\in\N$ and let $(x_k)_0^{2^n}$ be an equidistant partition of $\T$ of size $\frac{1}{2^n}$ with $x_0=x_{2^n}$. For any $k\in \{1,\dots,2^n\}$, we define the open interval $U_k:=]x_{k-1};x_k[$ of $\T$. It is immediate to see that $\{\overline{U_k}\}_1^{2^n}$ is a (finite) closed cover of $\T$. Now define
\[
\Lambda_n:=\{f\in\Lsc(\T,\{0,1\})\mid f_{|U_k} \text{ is constant for any }k\in \{1,\dots, 2^n\}\}.
\]
For any two $\Cu$-morphisms $\alpha, \beta:\Lsc(\T,\overline{\N})\longrightarrow T$ we define
\[dd_{\Cu}(\alpha,\beta):= \inf\limits_{n\in\N} \{ \frac{1}{2^n} \mid \alpha\underset{\Lambda_n}{\simeq}\beta \}.
\]  
If the infimum does not exist, we set the value to $\infty$.
We refer to $dd_{\Cu}$ as the \emph{discrete $\Cu$-semimetric}.
\end{prg}
Throughout the paper, we will make use of the discrete $\Cu$-semimetric for practical reasons. We recall that $dd_{\Cu}$ and $d_{\Cu}$ are topologically equivalent as shown in \cite[Proposition 5.5]{C22}. More specifically, one has that $dd_{\Cu}\leq d_{\Cu}\leq 2dd_{\Cu}$ and hence, $dd_{\Cu}$ satisfies the 2-relaxed triangle inequality, in the sense that $dd_{\Cu}(\alpha,\gamma)\leq 2(dd_{\Cu}(\alpha,\beta)+dd_{\Cu}(\beta,\gamma))$ for any $\alpha,\beta,\gamma\in \Hom_{\Cu}(\Lsc(\T,\overline{\N}),T)$ where $T$ is any $\Cu$-semigroup.

Furthermore, for any two unitary elements $u,v$ of a unital $\CatCa$-algebra $A$, one always has \[d_{\Cu}(\Cu(\varphi_u),\Cu(\varphi_v)),dd_{\Cu}(\Cu(\varphi_u),\Cu(\varphi_v))\leq d_U(\varphi_u,\varphi_v)\]
where $d_U(\varphi_u,\varphi_v):=\inf\limits_{w\in\mathcal{U}(A)}\sup\limits_{f\in\rm Lip(\T)} \Vert w\varphi_u(f)w^*-\varphi_v(f)\Vert$.

We end these preliminaries with some key facts about the de la Harpe-Skandalis determinant, introduced in \cite{dHS84}. This determinant will play a key role in providing obstructions.\\

$\hspace{-0,34cm}\bullet\,\,\textbf{De la Harpe-Skandalis determinant}.$ 
We refer the reader to \cite{dHS84} for details on this determinant and we briefly recall some facts that will be needed later on. For a unital $\CatCa$-algebra $A$, we denote the subgroup of $A$ generated by additive commutators $[a,b]:=ab-ba$ by $[A,A]$ and we define the \emph{universal trace on $A$} to be the quotient map $\Tr:A\longrightarrow A/\overline{[A,A]}$. We extend this (continuous) tracial linear map to $M_{\infty}(A)$ in a canonical way that we also denote $\Tr:M_{\infty}(A)\longrightarrow A/\overline{[A,A]}$. Now for any smooth path $\xi:[0,1]\longrightarrow \mathcal{U}^0_{\infty}(A)$, we define 
\[
\hat{\Delta}_{\Tr}(\xi):=\frac{1}{2i\pi}\int_0^1 \Tr(\xi'(t)\xi^{-1}(t))dt
\]
and the application $\hat{\Delta}_{\Tr}$ naturally induces a group morphism $\underline{\Tr}:\K_0(A)\longrightarrow A/\overline{[A,A]}$ by sending $[e]\longmapsto \Tr(e)$ for any projection $e\in M_{\infty}(A)$. (This is done via the identification $\K_0(A)\simeq \pi_1(\mathcal{U}^0_{\infty}(A))$.) We now define a group morphism called the \emph{de la Harpe-Skandalis determinant} as follows
\[
\begin{array}{ll}
\Delta_{\Tr}:\mathcal{U}^0_{\infty}(A)\longrightarrow \cokernel(\underline{\Tr})\\
\hspace{1,9cm} u\longmapsto [\hat{\Delta}_{\Tr}(\xi_u)]
\end{array}
\]
where $\xi_u$ is any smooth path connecting the identity to $u$. In particular, $\Delta_{\Tr}(e^{2i\pi h})=[\Tr(h)]$ for any self-adjoint $h\in M_{\infty}(A)$. 

As proven in the seminal work of de la Harpe and Skandalis, we have $D\mathcal{U}^0_{\infty}(A)\subseteq\ker \Delta_{\Tr}\subseteq\overline{\ker \Delta_{\Tr}}=\overline{D\mathcal{U}^0_{\infty}(A)}$ (see \cite[Proposition 4]{dHS84}). Nevertheless, these inclusions may be strict in some specific cases. A fortiori,  $\Delta_{\Tr}$ is not continuous in general. To remedy this, Thomsen introduced the \emph{non-stable de la Harpe-Skandalis determinant}
 \[
 \overline{\Delta}_{\Tr}\colon \mathcal{U}^0_{\infty}(A)\longrightarrow (A/\overline{[A,A]})/\overline{\im(\underline{\Tr})}
 \]
to be the composition of $\Delta_{\Tr}$ with the canonical projection of $\cokernel(\underline{\Tr})\relbar\joinrel\twoheadrightarrow  (A/\overline{[A,A]})/\overline{\im(\underline{\Tr})}$. It is shown in \cite[Theorem 3.2]{T95} that $\overline{\Delta}_{\Tr}$ is a continuous group morphism whose kernel is $\overline{D\mathcal{U}^0_{\infty}(A)}$. 

This is why it is often useful to consider the non-stable version of the de la Harpe-Skandalis  determinant. We end these preliminaries by the following proposition that be used throughout the manuscript and illustrates the force of the (non-stable) de la Harpe-Skandalis determinant.

\begin{prop}
\label{prop:dhs}
Let $A$ be a unital $\CatCa$-algebra. Let $u,v\in A$ be unitary elements connected to the identity. Then $\overline{\Delta}_{\Tr}(u)=\overline{\Delta}_{\Tr}(v)$ whenever $u\sim_{aue} v$.
\end{prop} 

\begin{proof}
 Observe that for any unitary element $w\in A$, the element $wuw^*u^*$ belongs to  the subgroup $D\mathcal{U}(A)$ of $\mathcal{U}(A)$ generated by multiplicative commutators $(u,v):=uvu^{-1}v^{-1}$. Moreover using a well-known trick, namely writing $\begin{psmallmatrix}uvu^{-1}v^{-1}\\&1 \\&&1\end{psmallmatrix}=\begin{psmallmatrix}u\\&u^{-1} \\&&1\end{psmallmatrix}\begin{psmallmatrix}v\\&1 \\&&v^{-1}\end{psmallmatrix}\begin{psmallmatrix}u^{-1}\\&u \\&&1\end{psmallmatrix}\begin{psmallmatrix}v^{-1}\\&1 \\&&v\end{psmallmatrix}$
we deduce that $D\mathcal{U}(A)\subseteq D\mathcal{U}_0(M_3(A))\subseteq D\mathcal{U}^0_{\infty}(A)$.  

Let $(w_n)_n$ be a sequence of unitary elements in $A$ such that $w_nuw_n^*\underset{n\rightarrow\infty}\longrightarrow v$. Then $\lim\limits_n w_nuw_n^*u^*$ belongs to $\overline{D\mathcal{U}(A)}\subseteq\overline{D\mathcal{U}^0_{\infty}(A)}=\ker\overline{\Delta}_{\Tr}$. (The second step is given by \cite[Theorem 3.2]{T95}.) 

We obtain that $\overline{\Delta}_{\Tr}(\lim\limits_n w_nuw_n^*u^*)=0$, that is, $\overline{\Delta}_{\Tr}(\lim\limits_n w_nuw_n^*)+\overline{\Delta}_{\Tr}(u^*)=0$. We conclude $\overline{\Delta}_{\Tr}(v)=\overline{\Delta}_{\Tr}(u)$.
\end{proof}

\section{Unitary Elements of AF algebras}
This section is dedicated to the existence part of the classification of unitary elements of unital $\AF$-algebras. Since $\CC(\T)$ is semiprojective, it is natural (and likely necessary) to impose a similar property on the abstract $\Cu$-morphisms we are aiming to lift. More specifically, we define a notion of \emph{weakly semiprojective $\Cu$-morphisms} and we prove that for any unital $\AF$-algebra $A$ and any weakly semiprojective $\Cu$-morphism $\alpha:\Lsc(\T,\overline{\N})\longrightarrow \Cu(A)$ such that $\alpha(1_\T)=[1_A]$, we can find a *-homomorphism $\phi:\CC(\T)\longrightarrow A$ lifting $\alpha$, i.e. such that $\Cu(\phi)=\alpha$. Our approach parallels the uniqueness part: we first prove the result for the finite-dimensional case, then extend it to the general case using weak semiprojectivity. \\

$\hspace{-0,34cm}\bullet\,\,\textbf{The finite dimensional case}.$ 
As a first step, let us focus on the finite dimensional case. We show that we can lift any abstract $\Cu$-morphism, up to arbitrary precision, and the desired lift will then be built combining the uniqueness result. It is worth noting that no specific properties are assumed for the $\Cu$-morphisms in this setting.

Before diving into the core of the matter, let us give an outlook of the proof. The strategy involves constructing a unitary element in the finite-dimensional $\CatCa$-algebra $B$. To do so, we first look at the values of the $\Cu$-morphism evaluated at (indicator maps of) some well-chosen connected arcs obtained from the canonical finite closed cover $\{\overline{U_k}\}_1^{2^n}$ of $\T$. These values are being used to \textquoteleft fill-up\textquoteright\ a diagonal matrix with entries in $\T$ which defines a unitary element of $B$, whose induced *-homomorphism gives rise to $\Cu$-morphism whose distance with $\alpha$ is bounded by $\frac{1}{2^n}$. 

\begin{thm}
\label{thm:existencefd}
Let $B$ be a finite-dimensional $\CatCa$-algebra. Let $\alpha\colon\Lsc(\T,\overline{\N})\longrightarrow \Cu(B)$ be a $\Cu$-morphism such that $\alpha(1_\T)=[1_B]$.

For any $n\in\N$, there exists a unitary element $u_n$ in $B$ such that $ \alpha\underset{\Lambda_n}\simeq \Cu(\varphi_{u_n})$.
\end{thm}

\begin{proof}
Let $n\in\N$. We first show the result assuming that $B:=M_d(\C)$ for some $d\in\N$ and the general case will follow. Recall that $\Cu(M_d(\C))\simeq \N$ and we will use this identification throughout. Also, we denote $\{x_k\}_0^{2^n}$, with $x_0=x_{2^n}$, the equidistant partition of $\T$ that induces $\Lambda_n$. For each $1\leq k\leq 2^n$, we define $c_k$ to be the center of the open arc $U_k$ and $V_k:=\Int(\overline{U_k}\cup\overline{U_{k+1}})$. Equivalently, $V_k:=U_k\cup\{x_k\}\cup U_{k+1}$. (With convention $U_{2^{n+1}}=U_1$.)

We will recursively build a unitary element of $B$ that has the required properties. Let $u$ be an empty matrix. Note that $\alpha(1_W)\in\N$ for any open set $W\subseteq \T$. For any $1\leq k\leq 2^n$, we write 
\[q_k:= \alpha(1_{U_k}) \hspace{1cm} r_k:=\alpha(1_{V_k})-(\alpha(1_{U_k})+\alpha(1_{U_{k+1}}))
\]
and we \textquoteleft fill-up\textquoteright\ $u$ as follows:

(1) $u=u\oplus\diag_{q_k}(c_k)$. 

(2) If $r_k>0$, then $u=u\oplus\diag_{q_k}(x_k)$. (Observe that $r_k\geq 0$.)\\
In other words, we have 
\[
u= \begin{psmallmatrix}
\ddots \\
&\diag_{q_k}(c_k)\\
&&\diag_{r_k}(x_k)\\
&&&\ddots
\end{psmallmatrix}.
\]
The matrix $u$ has entries in $\T$ and hence $u$ is a unitary element of some matrix algebra $M_{d'}(\C)$. Let us write $\beta:=\Cu(\varphi_u)$. We aim to show that $d=d'$ and that $\alpha\underset{\Lambda_n}\simeq\beta$. 

To do so, we first deduce from the construction of $u$ that $\alpha(1_{U_k})=\beta(1_{U_k})$ and $\alpha(1_{V_k})=\beta(1_{V_k})$ for any $1\leq k\leq 2^n$. 
Now let $W$ be an open connected arc of the form $W:=U_{l}\cup(\bigcup\limits_{k=l+1}^{r-1}\overline{U_{k}})\cup U_{r}$.
We have \[1_W+(\sum_{k=l+1}^{r-1}1_{U_{k}})=\sum_{k=l}^{r-1}1_{V_k} \text{ in }\Lsc(\T,\overline{\N}).\]
Applying the $\Cu$-morphisms $\alpha$ and $\beta$, we get the following equalities in $\N$
\[
\left\{
\begin{array}{ll}
\alpha(1_{W})+\sum_{l+1}^{r-1}\alpha(1_{U_k})=\sum_{l}^{r-1}\alpha(1_{V_k}).\\
\beta(1_{W})+\sum_{l+1}^{r-1}\beta(1_{U_k})=\sum_{l}^{r-1}\beta(1_{V_k}).
\end{array}
\right.
\] 
Since $\alpha$ and $\beta$ agree on each $V_k$, we deduce that  
\[
\alpha(1_{W})+\sum_{l+1}^{r-1}\alpha(1_{U_{k}})=\beta(1_{W})+\sum_{l+1}^{r-1}\beta(1_{U_{k}})\text{ in } \N.
\]
Finally, using cancellation in $\N$ and the fact that $\alpha$ and $\beta$ also agree on each $U_k$,  we conclude that $\alpha(1_W)=\beta(1_W)$. In other words, $\alpha$ and $\beta$ agree on each $g\in\Lambda_n$ with connected support. 
Therefore, we deduce that $d'=d=[1_B]$ and that $\alpha$ and $\beta$ agree on $\Lambda_n$.

The general case now follows from the observation that any $\Cu$-morphism $\alpha\colon\Lsc(\T,\overline{\N})\longrightarrow \Cu(\bigoplus_1^n M_{d_k}(\C))$ with $\alpha(1_\T)=[1_B]$ can be decomposed into a direct sum of $\Cu$-morphisms \break$\alpha_k\colon\Lsc(\T,\overline{\N})\longrightarrow \Cu(M_{d_k}(\C))$ with $\alpha_k(1_\T)=[1_{M_{d_k}(\C)}]$, that is, $\alpha=(\alpha_k)_1^n$.
\end{proof}

\begin{cor}
\label{cor:existencefd}
Let $B$ be a finite-dimensional $\CatCa$-algebra. Let $\alpha:\Lsc(\T,\overline{\N})\longrightarrow \Cu(B)$ be a $\Cu$-morphism such that $\alpha(1_\T)=[1_B]$.
Then there exists a unitary element $u$ in $B$, unique up to approximate unitary equivalence, such that $ \alpha= \Cu(\varphi_{u})$. 
\end{cor}

\begin{proof}
We only have to prove the existence part of the statement. Let $n\in\N$. We know that there exists a unitary element $u_n$ in $B$ such that $\beta_n:=\Cu(\varphi_{u_n})\underset{\Lambda_{n}}\simeq \alpha$.  We obtain a sequence $(u_n)_n$ of unitary elements in $B$ and for any $n,m\in\N$ with $n\leq m$, we compute that
\begin{align*}
dd_{\Cu}(\beta_n,\beta_m)&\leq 2 (dd_{\Cu}(\beta_n,\alpha)+dd_{\Cu}(\alpha,\beta_m))\\
&\leq 2(\frac{1}{2^{n}}+\frac{1}{2^{m}})\\
&\leq \frac{1}{2^{n-2}}.
\end{align*}
By the uniqueness result obtained in \cite[Theorem 5.11]{C22}, we deduce that there exists a unitary element $w_{n,m}$ such that $\Vert w_{n,m}u_n w_{n,m}^* - u_m\Vert\leq \frac{1}{2^{n-2}}$. We deduce that  $d_U(u_n,u_m)\leq \frac{1}{2^{n-2}}$, i.e. the sequence $(u_n)_n$ is Cauchy with respect to the metric $d_U$. 
On the other hand, it is well-known (at least, not hard to prove) that the set of unitary elements of a unital $\CatCa$-algebra is complete with respect to $d_U$. Therefore, the sequence $(u_n)_n$ converges towards a unitary element $u\in B$. Let us write $\beta:=\Cu(\varphi_u)$. We have 
\vspace{-0,4cm}\begin{align*}
dd_{\Cu}(\alpha,\beta)&\leq 2 (dd_{\Cu}(\alpha,\beta_n)+dd_{\Cu}(\beta_n,\beta))\\
&\leq  2 dd_{\Cu}(\alpha,\beta_n)+ 2 d_U(u_n,u)\\
&\leq \frac{1}{2^{n-1}}+ 2 d_U(u_n,u)
\end{align*}
and we easily see that the right side of the inequality tends to zero as $n$ goes to $\infty$, which shows that $dd_{\Cu}(\alpha,\beta)=0$, or equivalently, that $\alpha=\beta$.
\end{proof}

$\hspace{-0,34cm}\bullet\,\,\textbf{The general case}.$ Before diving into the completion of the classification of unitary elements of unital $\AF$-algebras, we take the opportunity to slightly improve of the uniqueness result obtained in \cite[Corollary 5.12]{C22} as follows.

\begin{thm}
\label{thm:uniquenessAF}
Let $A$ be a unital $\AF$-algebra. There exists a constant $C>0$ such that for any unital *-homomorphisms $\phi,\psi\colon \mathcal{C}(\T)\longrightarrow A$, we have that $ d_{U}(\phi,\psi) \leq C d_{\Cu}(\Cu(\phi),\Cu(\psi))$.
\end{thm}

\begin{proof}
In the separable setting, this is essentially contained in the proof of \cite[Corollary 5.12]{C22}. For the general setting, it is enough to improve \cite[Lemma 5.7]{C22} to any inductive system in the category $\Cu$. (Instead of inductive systems indexed over $\N$.) This is readily done via the characterization of inductive limits in the category $\Cu$ exposed for instance in \cite[Lemma 3.8]{TV22}.
\end{proof}

The next step is to pass existence results that we have obtained in the finite dimensional case to inductive limits. One could hope that a similar method as done for uniqueness in \cite{C22} via a factorization result analogous to \cite[Lemma 5.7]{C22} would be satisfactory. Nevertheless, it may not be true in general, that abstract $\Cu$-morphisms $\alpha\colon\Lsc(\T,\overline{\N})\longrightarrow \lim\limits_{\longrightarrow} S_i$ can be approximated by a factorization through a finite stage. (Whereas for concrete $\Cu$-morphisms, this was made possible by the fact that unitary elements of $\lim\limits_{\longrightarrow} A_i$ can be approximated by unitary elements of a finite stage, or equivalently, that $\CC(\T)$ is semiprojective.) 

As a consequence, it is relevant to define a \emph{weak semiprojectivity property} for $\Cu$-morphisms analogous to that of weak semiprojectivity for *-homomorphisms. Let us mention that we have been inspired by the $\CatCa$-version exposed in \cite[Definition 2.1 - Remark 2.6]{T18}.

\begin{dfn}\label{dfn:wsp}
Let $S,T$ be $\Cu$-semigroups. We say that a $\Cu$-morphism $\alpha\colon S\longrightarrow T$ is \emph{weakly semiprojective} if, for any inductive system $(V_i,\nu_{ij})$ in $\Cu$ and any $\Cu$-morphism $\tau\colon T\longrightarrow \lim\limits_{\longrightarrow}(V_i,\nu_{ij})$ and any finite set $F\subseteq S$, there exists a finite index $i_F$ and a $\Cu$-morphism $\alpha_F\colon S\longrightarrow V_{i_F}$ such that $\nu_{i_F\infty}\circ \alpha_F\simeq_F \tau\circ \alpha$. 

In other words, the following diagram approximately commutes within $F$
\[
  \xymatrix@R=5mm{
     & && V_{i_F}\ar[d]^{\nu_{i_F\infty}} \\
S\ar@{-->}@/^{1pc}/[urrr]^{\alpha_F}\ar[r]_{\alpha}  & T\ar[rr]_{\tau} &\ar@{}[u]_{\simeq_F} &\lim\limits_{\longrightarrow}(V_i,\nu_{ij})
   } 
   \]
\end{dfn}

We finally combine this property  with \autoref{thm:existencefd} to obtain the following theorem.

\begin{thm}
\label{thm:classificationAF}
Let $A$ be a unital $\AF$-algebra. Let $\alpha\colon\Lsc(\T,\overline{\N})\longrightarrow \Cu(A)$ be a weakly semiprojective $\Cu$-morphism such that $\alpha(1_\T)=[1_A]$. 

Then there exists a unitary element $u$ in $A$, unique up to approximate unitary equivalence, such that $ \alpha= \Cu(\varphi_u)$. 
\end{thm}

\begin{proof}
We only have to prove the existence part of the statement. We know that there is an inductive system of finite dimensional $\CatCa$-algebras such that $(A,\phi_{i\infty})_i:=\lim\limits_{\longrightarrow}(A_i,\phi_{ij})$. We conveniently write $S_i:=\Cu(A_i)$, $\sigma_{ij}:=\Cu(\phi_{ij})$ and $\sigma_{i\infty}:=\Cu(\phi_{i\infty})$. Let $n\in\N$. By the weak semiprojectivity of $\alpha$, we can find $\alpha_n\colon \Lsc(\T,\overline{\N})\longrightarrow S_{i_n}$ such that $\sigma_{i_n\infty}\circ\alpha_n\underset{\Lambda_n}\simeq \alpha$. By \autoref{thm:existencefd}, there exists a unitary element $u_n$ in $A_{i_n}$ such that $\Cu(\varphi_{u_n})\underset{\Lambda_n}\simeq \alpha_n$. Even more so, we know that $\Cu(\varphi_{u_n})$ and $\alpha_n$ agree on $\Lambda_n$. This implies that $\sigma_{i_n\infty}\circ\Cu(\varphi_{u_n})\underset{\Lambda_n}\simeq \alpha$. Write $v_n:=\phi_{i_n\infty}(u_n)$ and $\beta_n:=\Cu(\varphi_{v_n})$. Then $(v_n)_n$ is a sequence of unitary elements of $A$ and we have $dd_{\Cu}(\beta_n,\alpha)\leq \frac{1}{2^n}$. We now compute
\begin{align*}
dd_{\Cu}(\beta_n,\beta_m)&\leq 2 (dd_{\Cu}(\beta_n,\alpha)+dd_{\Cu}(\alpha,\beta_m))\\
&\leq 2(\frac{1}{2^n}+\frac{1}{2^m})\\
&\leq \frac{1}{2^{n-2}}.
\end{align*}
Arguing similarly as in the proof of \autoref{cor:existencefd}, it follows from \autoref{thm:classificationAF} that the sequence $(v_n)_n$ is Cauchy with respect to $d_U$ and hence, converges towards a unitary element $v\in A$. We also similarly compute that $dd_{\Cu}(\alpha,\Cu(\varphi_v))\leq 2 dd_{\Cu}(\alpha,\beta_n)+ 2 d_U(v_n,v)\underset{n\rightarrow \infty}\longrightarrow 0$ and we conclude that $\Cu(\varphi_v)=\alpha$.
\end{proof}

With these tools, we can now complete the classification of unitary elements of unital AF-algebras using a mildly reformulated Cuntz semigroup. This mild reformulation relates to an adjustment of the domain and the codomain of the functor $\Cu$, which enables us to obtain the desired classification. \\

\textbf{[The category $\mathcal{T}\!\star\AF$]} 
Let $\mathcal{T}$ be the subcategory of $\CatCa$ consisting of a single object $\mathcal{C}(\T)$ together with the identity morphism. Let $\AF$ denotes the full subcategory of $\CatCa$ consisting of all $\AF$-algebras with their *-homomorphisms. We now define the \emph{$\CatCa$-join of $\mathcal{T}$ and $\AF$}, denoted by $\mathcal{T}\!\star\AF$, to be the category whose objects are either objects of $\mathcal{T}$ or objects of $\AF$ and morphisms are 
\[
\Hom_{\mathcal{T}\!\star\AF}(A,B):=\left\{ 
\begin{array}{ll} 
\Hom_{\mathcal{T}}(A,B) \text{ if }A,B\in \mathcal{T}\\
\Hom_{\mathcal{\AF}}(A,B) \text{ if }A,B\in \AF\\
\Hom_{\mathcal{\CatCa}}(A,B) \text{ if }A\in \mathcal{T} \text{ and } B \in \AF\\
\Hom_{\mathcal{\CatCa}}(A,B) \text{ if }A\in \AF  \text{ and } B \in \mathcal{T}
\end{array}
\right.
\]

That $\mathcal{T}\!\star\AF$ is a well-defined category is readily checked. We mention that the $\CatCa$-join construction has been inspired by the existing notion of \emph{join of categories}. See e.g. \cite[\S1.2.8]{L09}.\\

\textbf{[The category $\Cu_{\wsp}$]} 
Let us consider the category $\Cu_{\wsp}$ consisting of all $\Cu$-semigroups together with weakly semiprojective $\Cu$-morphisms and the identity morphism on each object. 

That $\Cu_{\wsp}$ is a well-defined wide subcategory of $\Cu$ follows after observing that the composition of $\Cu$-morphisms in $\Cu_{\wsp}$ is either weakly semiprojective -when both are weakly semiprojective or one is weakly semiprojective and the other is an identity morphism- or an identity morphism -when both are identity morphisms-. 

\begin{prop}
\label{rmk:lll}
Let $A,B\in \mathcal{T}\!\star\AF$ and let $\phi\colon A\longrightarrow B\in \Hom_{\mathcal{T}\star\AF}(A,B)$. 

Then $\Cu(\phi)$ is weakly semiprojective. 
\end{prop}

\begin{proof}
The only non-trivial case to cover is $\phi\colon \mathcal{C}(\T)\longrightarrow A$ with $A\in \AF$.  From the semiprojectivity of $\mathcal{C}(\T)$, we know that there exists an inductive system $(A_i,\psi_{ij})_i$ of finite dimensional $\CatCa$-algebras with limit $(A,\psi_{i\infty})_i$ and a lift $\tilde{\phi}\colon \mathcal{C}(\T)\longrightarrow A_i$ such that $\psi_{i\infty}\circ\tilde{\phi}=\phi$. 
As a consequence, we know that $\Cu(\phi)=\Cu(\psi_{i\infty})\circ\Cu(\tilde{\phi})$, that is, $\Cu(\phi)$  factors through $\Cu(A_i)$ which can be identified with $\Cu(A_i)\simeq\overline{\N}^{r}$ for some $r\in\N$. 

Now consider any inductive system $(T_i,\tau_{ij})_i$ in $\Cu$ with limit $(T,\tau_{i\infty})$ and let $\beta\colon \Cu(A)\longrightarrow T$ be any $\Cu$-morphism. Hence, it is enough to find a lift of $\beta\circ\Cu(\psi_{i\infty})\colon \overline{\N}^{r}\longrightarrow T$. Let $e_k$ denotes the compact element of $\Cu(A_i)$ that generates the $k$-th copy of $\overline{\N}$. Using the characterization of inductive limits in the category $\Cu$ (see e.g. \cite[Lemma 3.8]{TV22}), it is readily checked that $\beta\circ\Cu(\psi_{i\infty})$ can be realized by $\tau_{i_k\infty}(t_k)$ for some $t_k\in T_{i_k}$. Finally, since the finite $\{x_k\}_{k=1}^r$ entirely determines $\beta\circ\Cu(\psi_{i\infty})$, we can construct a $\Cu$-morphism $\tilde{\beta}\colon \Cu(A_i)\longrightarrow T_{l}$ for some $l\geq i_k$ for all $1\leq k\leq r$ such that $\tau_{l\infty}\circ \tilde{\beta}= \beta\circ\Cu(\psi_{i\infty})$. Precomposing with $\Cu(\tilde{\phi})$ yields the desired result.
\end{proof} 

\textbf{[The restricted functor $\Cu$]} As a consequence of all the above, we can define a restricted functor $\Cu\colon \mathcal{T}\!\star\AF\longrightarrow \Cu_{\wsp}$ and we obtain an immediate corollary.

\begin{cor}
Let $A$ be a unital $\AF$-algebra. Then the (scaled) functor $\Cu\colon \mathcal{T}\!\star\AF\longrightarrow \Cu_{\wsp} $ classifies unitary elements of unital $\AF$-algebras.
\end{cor}

\begin{proof}
By construction, we have that $\Hom_{\mathcal{T}\!\star\AF}(\CC(\T),A)=\Hom_{\CatCa}(\CC(\T),A)$ for any for any $\AF$-algebra. The result now follows from \autoref{thm:classificationAF} together with \autoref{rmk:lll}.
\end{proof}

\begin{rmk}
(i) A theory of (weakly) semiprojective $\Cu$-semigroups could be developed further, yet it does not lie in the scope of the paper and hence, we do not pursue this idea here. 

(ii) In our opinion, it is very unlikely that $\Cu(\CC(\T))$ is (weakly) semiprojective even though $\CC(\T)$ is. However, as of now, it is unclear whether any abstract $\Cu$-morphism $\alpha\colon\Lsc(\T,\overline{\N})\longrightarrow \Cu(A)$ is automatically weakly semiprojective whenever $A$ is \textquoteleft nice enough\textquoteright\, e.g., A is an $\AF$-algebra.

(iii) In the simple setting, one could use classification machinery (exposed e.g. \cite{CGSTW23}) to deduce that the \textquoteleft unrestricted\textquoteright\ functor $\Cu\colon\CatCa\longrightarrow \Cu$ classifies unitary elements of unital simple $\AF$-algebras.
\end{rmk}

\section{Unitary Elements of $\AH_1$-algebras}\label{sec:4}
In the sequel, we are interested in matrix algebras of continuous functions over any compact 1-dimensional $\CW$-complex, their direct sums and their inductive limits. These $\CatCa$-algebras are precisely the inductive limits of (direct sums of) homogeneous $\CatCa$-algebras with one-dimensional spectrum. Therefore, we refer to this class of $\CatCa$-algebras as \emph{$\AH_1$-algebras}. Observe that $\AH_1$-algebras contain $\AF,\AI$ and $\A\!\T$-algebras and always have a torsion-free $\K_1$-group. 

We divide this section in two parts. We first provide an obstruction to uniqueness by exhibiting two unitary elements in $\CC([0,1])\otimes M_{2^\infty}$ that agree at level of $\Cu$ and yet, fail to be approximately unitarily equivalent. Next, we explore the existence part of the classification. We show some satisfactory results, namely that any $\Cu$-morphism $\alpha:\Lsc(\T,\overline{\N})\longrightarrow \Cu(A)$ such that $\alpha(1_\T)=[1_A]$, where $A$ is any $\AH_1$-algebra, can be approximately lifted up to arbitrary precision. By introducing a notion of \emph{Cauchy sequences}, we are able to improve the latter result by exposing a Cauchy sequence $(\Cu(\varphi_{u_n}))_n$ of concrete $\Cu$-semigroup converging towards $\alpha$ with respect to $dd_{\Cu}$. However, achieving an existence result via our methods requires a uniqueness argument, which does not hold here, as illustrated by the following example.

\subsection{Obstruction to Uniqueness}
\label{subsec:4A}
Let $A:=\CC([0,1])\otimes M_{2^\infty}$. In order to define the desired unitary elements of $A$, we first construct a unitary element $w$ of $M_{2^\infty}$. Recall that $M_{2^\infty}$ is constructed as the inductive limit of the inductive sequence $(M_{2^n},\phi_{nm})_n$ where $\phi_{nn+1}:M_{2^n}\longrightarrow M_{2^{n+1}}$ sends $a\longmapsto \begin{psmallmatrix}a \\&a \end{psmallmatrix}$. 
Now for any $n\in \N$ we consider a unitary element in $M_{2^n}$ defined by 
\[w_n:= \begin{psmallmatrix}1 \\&e^{2i\pi/2^n} \\&&\ddots\\&&&e^{2i\pi(2^n-1)/2^n}\end{psmallmatrix}\]  
It is not hard to compute that $d_U(w_m,\phi_{nm}(w_n))\leq1/2^n-1/ 2^m$. (Roughly speaking, allowing unitary equivalence allows to re-arrange the diagonal terms.) We deduce that the sequence $(\phi_{n\infty}(w_n))_n$ is Cauchy with respect to the metric $d_U$. Similarly as before, we know that the set of unitary elements of $M_{2^\infty}$ is complete with respect to $d_U$. Therefore, the latter sequence converges towards a unitary $w\in M_{2^\infty}$ whose spectrum is $\T$. We now define unitary elements of $A$ as follows \[u:=1_{[0,1]}\otimes w \text{ \hspace{1cm}and\hspace{1cm} } v:=e^{2i\pi\id_{[0,1]}}\otimes \,w\] 
We first aim to show that $u$ and $v$ agree at level of $\Cu$. To do so, we introduce for any $n\in\N$
\[u_n:=1_{[0,1]}\otimes w_n \text{ \hspace{1cm}and\hspace{1cm} } v_n:=e^{2i\pi\id_{[0,1]}}\otimes \,w_n\] 
For notational purposes, let us write $\alpha:=\Cu(\varphi_u)$ and $\beta:=\Cu(\varphi_v)$. Respectively $\alpha_{n}:=\Cu(\varphi_{u_n})$ and $\beta_{n}:=\Cu(\varphi_{v_n})$. 

Let $n\in\N$. We highlight that for an open set $U$ of $\T$, the image of the indicator map $1_U$ by $\alpha_{n},\beta_{n}$ evaluated at some $t\in [0,1]$ corresponds to the number of eigenvalues of $u_n,v_n$ contained in the \textquoteleft space\textquoteright\ $U\subseteq \T$ at \textquoteleft time\textquoteright\ $t\in [0,1]$. Therefore,
we compute that for $g,h\in \Lambda_n$ such that $g\neq 1_\T$ and $g\ll h$, and any $t\in [0,1]$
\[
\left\{
\begin{array}{ll}
\alpha_{n}(g)(0)=\beta_{n}(g)(0) \\
\alpha_{n}(g)(t)\leq\alpha_{n}(g)(0)+k\leq \alpha_n(h)(t)\\
\beta_n(g)(t)\leq \beta_n(g)(0)+k\leq\beta_n(h)(t) 
\end{array}
\right.
\]
where $k$ is the number of connected components of $\supp g$. This leads to
\[
\left\{
\begin{array}{ll}
\alpha_n(g)(t)\leq  \beta_n(h)(t)\\
\beta_n(g)(t)\leq \alpha_n(h)(t)
\end{array}
\right.
\]
For the particular case $g=1_\T$, we have $\alpha_{n}(g)(t)=\beta_{n}(g)(t)=2^n$ for all $t\in [0,1]$. It follows that $dd_{\Cu}(\alpha_{n},\beta_{n})\leq 1/2^n$. But we also know that $d_U((\id\otimes\phi_{n\infty})(u_n),u),d_U((\id\otimes\phi_{n\infty})(v_n),v)\leq 1/2^n$ from which we deduce that \[dd_{\Cu}(\alpha,\beta)=0\] using standard arguments. Equivalently $u$ and $v$ agree at level of $\Cu$.
However, it seems almost inevitable that they are not close with respect to $d_U$. The intuition behind the fact that $u$ and $v$ are far apart with respect to $d_U$ relies on the fact that $v$ has \textquoteleft movement\textquoteright\ induced by $\id_{[0,1]}$, while $u$ \textquoteleft stays put\textquoteright\ due to $1_{[0,1]}$. This \textquoteleft movement information\textquoteright\ can be captured by the de la Harpe-Skandalis determinant. 
We recall that $A$ is isomorphic to $\CC([0,1],M_{2^\infty})$. It follows that $A/\overline{[A,A]}$ is isomorphic to $\CC([0,1],M_{2^\infty}/\overline{[M_{2^\infty},M_{2^\infty}]})$. From now on, we abusively use this identifications, in the sense that we picture elements of $A$ as continuous functions from $X$ to $M_{2^\infty}$. (Similarly for elements of $A/\overline{[A,A]}$.) 
 On the other hand, it is well-known that $M_{2^\infty}$ has a unique trace $\tau_M$ which happens to be the universal trace on $M_{2^\infty}$. A fortiori $M_{2^\infty}/\overline{[M_{2^\infty},M_{2^\infty}]}\simeq \C$. Therefore, the universal trace $\Tr$ on $A$ is sending any $f\in \CC([0,1],M_{2^\infty})$ to 
\[
\begin{array}{ll}
\Tr(f):[0,1]\longrightarrow \C\\
\hspace{1,35cm}t\longmapsto \tau_M(f(t))
\end{array}
\]
Moreover $f\in\CC([0,1],M_{2^\infty})$ is a projection if (and only if) $f$ is constant with value a projection of $M_{2^\infty}$ and $\K_0(M_{2^\infty})\simeq \Z[\frac{1}{2}]$. We hence deduce that the de la Harpe-Skandalis determinant $\Delta_{\Tr}$ has codomain $\CC([0,1],\C/\Z[\frac{1}{2}])$ and its non-stable version $\CC([0,1],\C)/\{k.1_{[0,1]}\mid k\in \R\}$.

Finally, observe that there exists a self-adjoint element $h\in M_{2^\infty}$ such that $e^{2i\pi h}=w$. (Take the limit of the Cauchy sequence $\phi_{n\infty}(\diag_n(1, 1/2^n,\dots,(2^n-1)/2^n))$.) We compute that 
\[
\Delta_{\Tr}(u): t\longmapsto [\tau_M(h)]_{\C/\Z[\frac{1}{2}]} \text{ \hspace{1cm}and\hspace{1cm} } \Delta_{\Tr}(v):t\longmapsto [\tau_M(h)+t]_{\C/\Z[\frac{1}{2}]}
\]
It is now obvious that $u$ and $v$ have distinct (non-stable) de la Harpe-Skandalis determinants, which implies that they are not approximately unitarily equivalent. (See \autoref{prop:dhs}.)

We remark that similar examples can be easily constructed in $\CC(X)\otimes M_{2^\infty}$ where $X$ is any compact Hausdorff space of dimension 1.

\subsection{On the path towards Existence}
We focus now on the existence part of the classification. Namely, given a unital $\AH_1$-algebra $A$, we investigate whether any $\Cu$-morphism $\alpha:\Cu(\CC(\T))\longrightarrow \Cu(A)$ with $\alpha(1_\T)=[1_A]$ lifts to a *-homomorphism $\phi_{\alpha}:\CC(\T)\longrightarrow A$. 
As for the $\AF$-case, we first have a look at \textquoteleft building blocks\textquoteright. In other words, we consider $\CatCa$-algebras that are (finite) direct sums of matrix algebras of the form $M_d(\CC(X))$, where $X$ is any compact 1-dimensional $\CW$ complex. 

Our strategy involves splitting $X$ into (finitely many) pieces on which the restriction of the abstract $\Cu$-morphism behaves similarly to the finite dimensional case.
We then apply the previous lifting result to each piece, obtaining (finitely many) constant matrices with entries in $\T$. Subsequently, we ought to carefully connect continuously each of these matrices to build a unitary element of $B$ with the desired properties.

Before diving into the main proof, let us present a couple of lemmas that will be of use when cutting $X$ into pieces and gluing continuously the matrices associated to each piece.

\begin{ntn}
Let $X$ be a compact Hausdorff space.

(i)  Let $U$ be an (open) set in $X$. For any $\delta>0$ we write $U_\delta:=\bigcup\limits_{x\in U}\mathcal{B}(x,\delta)$. We also write $\Int_\delta(U)$ to be the largest open set $V$ contained in $U$ such that $V_\delta\subseteq U$. Note that for $\dim X=1$, we always have $(\Int_\delta(U))_\delta=U$.

(ii) Let $f\in \Lsc(X,\overline{\N})$. The chain-decomposition of $f$ is given by $(1_{f^{-1}(]n,\infty])})_{n\in\N}$ and denoted by $\dcp(f)$. (See \cite[Section 4]{C22} for more details.) 
In the sequel, it will be convenient to reindex it as follows: $\dcp(f)=(1_{W_n})_{n\geq 1}$ where $W^f_{n}:=f^{-1}(]n-1,\infty])$ for all $n\geq 1$. In particular, we get that $f_{\mid (W_i^f\setminus W_{i+1}^f)}=i$. 
\end{ntn}

\begin{lma}
\label{lma:cut}
Let $X$ be a compact 1-dimensional $\CW$ complex and let $f_1,\dots,f_p$ be elements in $\Lsc(X,\overline{\N})_{\ll}$. Then there exists a (maximal) finite closed cover $\{\overline{T_l}\}_1^m$ of $X$ such that the restrictions ${f_i}_{\mid T_l}$ are constant for all $1\leq i\leq p$ and all $1\leq l\leq m$. 
\end{lma}

\begin{proof}
Recall that each $f_i$ has a unique chain-decomposition $\dcp(f_i)=(W_1^i,\dots, W_{m_i}^i)$ where $m_i=\max\limits_{t\in X} \{f_i(t)\}<\infty$.  We need to take into account all the possible values taken simultaneously by all the $f_i$. To do so, we consider the finite set $\mathfrak{L}:=\{(l_1,\dots,l_p)\}$ of $p$-tuples whose $i$-th coordinate  belongs to $\{0,\dots,m_i\}$. We point out that $\card(\mathfrak{L}):=\prod\limits_{i=1}^p (m_i+1)$ and thus we can index the elements of $\mathfrak{L}$. Let $(l_1,\dots,l_p)$ to be the $l$-th element of $\mathfrak{L}$. We define the open set $T_l$ of $X$ associated to $(l_1,\dots,l_p)$ as follows:
\[
T_l:=\bigcap\limits_{i=1}^p \Int(W^i_{l_i-1}\setminus W^i_{l_i}) \text{ where } W_0^i:=\Int(X\setminus W_1^i).
\]
By construction, we have that ${f_i}_{\mid T_l}=l_i$ for all $1\leq i\leq p$. Also that $T_{l}$ is maximal, in the sense that for any other open set $T'$ such that ${f_i}_{\mid T'}=l_i$ for all $1\leq i\leq p$, then $T'\subseteq T_{l}$. Consequently, we deduce that $\{\overline{T_{l}}\}_{1}^{\card{\mathfrak{L}}}$ is a finite closed cover of $X$ with the desired properties.
\end{proof}

\begin{lma}
\label{lma:glue}
Let $X$ be a compact 1-dimensional $\CW$ complex. Let $f,g$ be elements in $\Lsc(X,\overline{\N})$ such that $f\ll g$ and let $W_l\in\dcp(f)$ be the $l$-th element in the chain-decomposition of $f$, respectively $Z_l\in\dcp(g)$ for $g$. 

Then, for any open set $U\subseteq W_l$ there exists $\delta>0$ such that $U_\delta\subseteq Z_l$.
\end{lma}

\begin{proof}
By \cite[Theorem 4.4]{C22}, we know that $W_l$ is compactly contained in $Z_l$ for each $l\in\N$. (Remark that $W_l$ is ultimately empty.) Equivalently, $\overline{W_l}\subseteq Z_l$ from which we deduce that there exists $\delta>0$ such that ${(W_l)}_\delta\subseteq Z_l$. The lemma follows after observing that $U_\delta\subseteq {(W_l)}_\delta$ for any $\delta>0$.
\end{proof}

Let us now prove the following approximate lifting result in the finite dimensional case.
\begin{thm}
\label{thm:existencefx}
Let $B$ be a direct sum of matrix algebras of continuous functions over 1-dimensional compact $\CW$-complexes. Let $\alpha:\Lsc(\T,\overline{\N})\longrightarrow \Cu(B)$ be a $\Cu$-morphism such that $\alpha(1_\T)=[1_B]$.

For any $n\in\N$, there exists a unitary element $u_n$ in $B$ such that $ \alpha\underset{\Lambda_{n-2}}\simeq \Cu(\varphi_{u_n})$.
\end{thm}

\begin{proof}
We first prove the theorem for $B=M_d(\CC(X))$, where $X$ is any compact 1-dimensional $\CW$ complex. Let $n\in\N$. For each $1\leq k \leq 2^n$, we write $V_k:=\Int(\overline{U_k}\cap\overline{U_{k+1}})$. \\

\emph{Step 1: We cut $X$ into finitely-many carefully-chosen pieces on which we can apply \autoref{thm:existencefd}.} 

By \autoref{lma:cut} there exists a (maximal) finite closed cover $\{\overline{T_l}\}_{1}^m$ such that $\alpha(g)_{\mid T_l}$ is constant for all $g\in\Lambda_n$ and all $1\leq l\leq m$. Consequently, for each $1\leq l\leq m$, we can mimic the construction done in the proof of \autoref{thm:existencefd} (pluging the integers $\alpha(1_{U_k})_{\mid T_l}$, $\alpha(1_{U_{k+1}})_{\mid T_l}$ and $\alpha(1_{V_k})_{\mid T_l}$ into the \textquoteleft algorithm\textquoteright) to obtain a diagonal unitary matrix $u_l\in M_d(\T)$ associated to $T_l$.\\

\emph{Step 2: We show via the Hall's marriage theorem that there exists a bijective correspondence $\sigma_{l,l'}$ between the eigenvalues of $u_l$ and $u_{l'}$ satisfying $\Vert \id - \sigma_{l,l'}\Vert <\frac{2}{2^n}$, whenever $\overline{T_l}\cap \overline{T_{l'}}\neq \emptyset$.}

 For any $1\leq l\leq m$, we write $X_l$ to be the set of eigenvalues of $u_l$. Fix $l,l'$ such that $\overline{T_l}\cap \overline{T_{l'}}\neq \emptyset$ and consider the bipartite graph $G_{l,l'}:=(X_l+X_{l'}, E_{l,l'})$ where
\[
E_{l,l'}:=\{(\lambda, \nu)\in X_l\times X_{l'} \mid \Vert \lambda - \nu\Vert < \frac{2}{2^n}\}.
\]
Let $\Omega\subseteq X_l$ and let $g_\Omega$ the minimal element of $\Lambda_n$ such that $\Omega\subseteq \supp(g_\Omega)$. We write $h_\Omega:=\min\{h\in\Lambda_n \mid g_\Omega\ll h\}$. (The existence of $g_\Omega$ and $h_\Omega$ can be justified by an explicit construction that we leave to the reader.) Observe that $T_l\subseteq \alpha(g_\Omega)^{-1}(]n_l;\infty])$, where $n_l:=\alpha(g_\Omega)_{\mid T_l}-1$. Thus, by \autoref{lma:glue}, there exists $\delta_{l,l'}>0$ such that ${(T_l)}_{\delta_{l,l'}}\subseteq \alpha(h_\Omega)^{-1}(]n_l;\infty])$. Further,  $\overline{T_l}\cap \overline{T_{l'}}\neq \emptyset$ implies that ${(T_l)}_{\delta_{l,l'}}\bigcap T_{l'}\neq \emptyset$. Therefore, there exists $t\in T_{l'}$ such that $\alpha(g_\Omega)_{\mid T_l}=n_l+1\leq \alpha(h_\Omega)(t)=\alpha(h_\Omega)_{\mid T_{l'}}$. On the other hand
\[
	\left\{
	   	\begin{array}{ll}
   			\alpha(g_\Omega)_{\mid T_l}=\card(X_l\cap \supp g_\Omega)=\card \Omega	\\
   			\alpha(h_\Omega)_{\mid T_{l'}}=\card(X_{l'}\cap \supp h_\Omega)
   		\end{array}
	\right.
\]
Lastly, we remark that for any $U_k\subseteq \supp g_\Omega$, we have that $\overline{U_k}\cap \Omega \neq \emptyset$. (By minimality of $g_\Omega$.) It follows that for any element $\nu\in \supp h_\Omega$, there exists an element $\lambda\in X_l\cap \supp g_\Omega$ such that  $\Vert \nu-\lambda\Vert<\frac{2}{2^n}$. In particular, this is true for any $\nu\in \supp h_\Omega\cap X_{l'}$. 

Thus, $X_{l'}\cap \supp h_\Omega\subseteq e_{G_{l,l'}}(\Omega)$ and we finally conclude that 
\[
\card\Omega=\alpha(g_\Omega)_{\mid T_l}\leq \alpha(h_\Omega)_{\mid T_{l'}}\leq \card e_{G_{l,l'}}(\Omega).
\]
By Hall's marriage theorem, see e.g. \cite[Theorem 5.10]{C22}, we obtain a bijective correspondence $\sigma_{l,l'}:X_l\overset{\simeq}\longrightarrow X_{l'}$ such that $\Vert \lambda - \sigma_{l,l'}(\lambda) \Vert < \frac{2}{2^n}$ for any $\lambda\in X_l$. \\

\emph{Step 3: We construct diagonal smooth path matrices $p_{l,l'}:[0,1]\longrightarrow M_d(\T)$ connecting $u_l$ to $u_{l'}$ satisfying $\Vert p_{l,l'}(t) - p_{l,l'}(0)\Vert < \frac{2}{2^n}$ for all $t\in [0,1]$, whenever $\overline{T_l}\cap \overline{T_{l'}}\neq \emptyset$.} 

Fix $l,l'$ such that $\overline{T_l}\cap \overline{T_{l'}}\neq \emptyset$ and let $\lambda\in X_l$. We can build a smooth path $p_{l,l'}^\lambda:[0,1]\longrightarrow \T$ that connects $\lambda$ to $\sigma_{l,l'}(\lambda)$ and such that $p_{l,l'}^\lambda(t)\in \mathcal{B}(\lambda,\frac{2}{2^n})$. We subsequently define a diagonal  \textquoteleft path matrix\textquoteright\ $p_{l,l'}:=\diag(p_{l,l'}^\lambda\mid \lambda\in X_l)$.\\

\emph{Step 4: We construct a unitary $u$ in $B$ piecewisely via the matrices $u_l$ and $p_{l,l'}$.}

Note that $B\simeq \CC(X,M_d)$ and hence, we will construct $u$ as a continuous map from $X$ to $M_d$. Let $x$ be a singularity point, i.e. $x\in\overset{m}{\underset{l=1}{\cup}} \partial(T_l)$. We first assume that there exists a unique pair of indices $l,l'$ such that $x\in\overline{T_l}\cap \overline{T_{l'}}$. (Equivalently, only two closed sets $\overline{T_l},\overline{T_{l'}}$ are meeting at $x$.)
 In particular, there exists a unique pair of connected component $L_x,R_x$ of $T_l,T_{l'}$ respectively such that $x\in\overline{L_x}\cap \overline{R_x}$. Now choose $\delta_{L_x},\delta_{R_x}>0$ small enough, e.g. $\delta_{L_x}<\diam (L_x)/8$, respectively $\delta_{R_x}<\diam (R_x)/8$. We define $C_x$ to be the largest open set that contains $x$ and does not intersect with $\Int_{\delta_{L_x}}(L_x)\cup \Int_{\delta_{R_x}}(R_x)$. In other words, $C_x$ is the connected component of $x$ in the set $(\overline{L_x}\setminus\Int_{\delta_{L_x}}(L_x)\cup (\overline{R_x}\setminus\Int_{\delta_{R_x}}(R_x)$. Let us illustrate these sets in the following picture. 
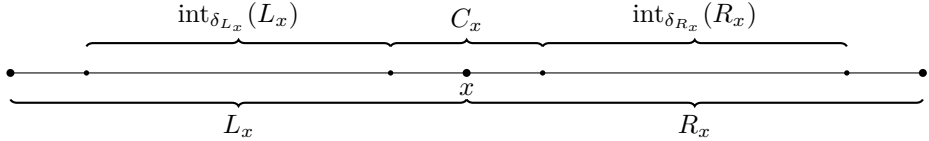
\vspace{0,5cm}

\hspace{1cm}\begin{tikzpicture} 
    \coordinate (1) at (0,0);
    \coordinate (2) at (1,0);
    \coordinate (3) at (5,0);
    \coordinate (4) at (6,0);
    \node at (4) [above=-12pt] {$x$};
    \coordinate (5) at (7,0);
    \coordinate (6) at (11,0);
    \coordinate (7) at (12,0);
   \draw[thick][decorate, decoration = { brace,mirror, raise=10pt}] (1) -- (4) node[pos=0.5,above=-28pt,black]{$L_x$};
   \draw[thick][decorate, decoration = { brace,mirror, raise=10pt}] (4) -- (7) node[pos=0.5,above=-28pt,black]{$R_x$};
   \draw[thick][decorate, decoration = { brace ,raise=10pt}] (3) -- (5) node[pos=0.5,above=12pt,black]{$C_x$};
   \draw[thick][decorate, decoration = { brace ,raise=10pt}] (2) -- (3) node[pos=0.5,above=12pt,black]{$\Int_{\delta_{L_x}}(L_x)$};
    \draw[thick][decorate, decoration = { brace ,raise=10pt}] (5) -- (6) node[pos=0.5,above=12pt,black]{$\Int_{\delta_{R_x}}(R_x)$};
    \draw (1) -- (7) ;
    \foreach \x in {(2), (3), (5), (6)}{\fill \x circle[radius=1pt];}
    \foreach \x in {(1), (4),  (7)}{\fill \x circle[radius=1.5pt];}
\end{tikzpicture}

Now construct a continuous map $u$ on $\Int_{\delta_{L_x}}(L_x)\cup\overline{C_x}\cup\Int_{\delta_{R_x}}(R_x)$ piecewisely as follows.
\[
\left\{
\begin{array}{ll}
	u(t)=u_l \text{ for all } t\in \Int_{\delta_{L_x}}(L_x)	\\
	u(t)=u_{l'} \text{ for all } t\in \Int_{\delta_{R_x}}(R_x)	\\
	u(t)=p_{l,l'}\circ i_x(t) \text{ for } t\in \overline{C_x} 
\end{array}
\right.
\]
where $i_x:\overline{C_x}\simeq [0,1]$ is a continuous map identifying $\overline{\Int_{\delta_{L_x}}(L_x)}\cap \overline{C_x}$ with $0$, $\overline{C_x}\cap \overline{\Int_{\delta_{R_x}}(R_x)}$ with $1$, and such that $i_x(x)=1/2$. 

In the general case, there might be (finitely many) other edges meeting at $x$. (Note that the singular point $x$ has to coincide with a vertex of $X$ for that to happen.) In other words, there might exists some $1\leq\hat{l}\leq m$ distinct from $l,l'$ such that $x\in\overline{T_l}\cap \overline{T_{\hat{l}}}$. Similarly as before, we define $Q_x$ to be the connected component of $T_{\hat{l}}$ such that $x\in \overline{Q_x}$ and $D_x$ to be the largest open set that contains $x$ and does not intersect with $\Int_{\delta_{L_x}}(L_x)\cup \Int_{\delta_{Q_x}}(Q_x)$. (See picture above.) However, we already have constructed the map $u$ on the set $\overline{L_x\cap C_x}$. Therefore, we are left to define $u$ on $\overline{D_x\cap Q_x}\cup\Int_{\delta_{Q_x}}(Q_x)$. We first define $u$ on $\Int_{\delta_{Q_x}}(Q_x)$ by
\[
	u(t)=u_{\hat{l}} \text{ for all } t\in \Int_{\delta_{Q_x}}(Q_x).	
\] 
Secondly, let us concatenate the smooth paths $p_{l,l'}$ and $p_{l,\hat{l}}$ to obtain the following diagonal smooth path matrix
\[
\begin{array}{ll}
	q_{l,\hat{l}}:[0,1]\longrightarrow M_d(\T)\\
					\hspace{1,3cm}t\longmapsto\left\{
									\begin{array}{ll}
									p_{l,l'}(1/2-t)\text{ for } t\in [0,1/2]\\
									p_{l,\hat{l}}(2t-1)\text{ for } t\in [1/2,1]
									\end{array}
									\right.
\end{array}
\] 
Observe that $q_{l,\hat{l}}(0)=u(x)$ and $q_{l,\hat{l}}(1)=u_{\hat{l}}$. Moreover, for each $\lambda\in X_l$, the path $q^\lambda_{l,\hat{l}}$ keeps satisfying that $q^\lambda_{l,\hat{l}}(t)\in \mathcal{B}(\lambda,\frac{2}{2^n})$ for all $t\in [0,1]$.

Lastly, we define $u$ on $\overline{D_x\cap Q_x}$ by
\[
	u(t)=q_{l,\hat{l}}\circ j_x(t) \text{ for } t\in \overline{D_x\cap Q_x} 
\] 
where $j_x:\overline{D_x\cap Q_x}\simeq [0,1]$ is a  map identifying $x$ with $0$ and $\overline{D_x}\cap \overline{\Int_{\delta_{Q_x}}(Q_x)}$ with $1$.

We highlight that during the time interval $\overline{D_x}$, the map $u$ continuously connects each $\lambda$ to $\sigma_{l,\hat{l}}(\lambda)$ and stays within $\mathcal{B}(\lambda,\frac{2}{2^n})$. Further, if more edges are meeting at $x$, they are dealt similarly as $\hat{l}$. 

Finally, in the particular case where only one index $l$ is meeting a given singular point $x$ (which implies that $x$ is not only a vertex of $X$ but also a sink) then we simply define $u(t):=u_l$ on the remaining piece.

In the end, we have built a continuous map $u:X\longrightarrow M_d(\T)$, i.e. a unitary element $u$ of $B$. \\

\emph{Step 5: We compute that $\Cu(\varphi_u)$ and $\alpha$ are close with respect to $dd_{\Cu}$}. 

Let us write $\beta:=\Cu(\varphi_u)$. We aim to prove that $\alpha\underset{\Lambda_{n-2}}\simeq \beta$. It is sufficient to show that $\alpha(g)\leq \beta(g_{3/2^n})$ and $\beta(g)\leq \alpha(g_{3/2^n})$ for any $g\in \Lambda_n$. Let $g\in\Lambda_n$ and let $x,L_x,C_x,R_x$ be as before. (Respectively $x,L_x,D_x,Q_x$ which is dealt similarly.)
Arguing similarly as in the proof of \autoref{thm:existencefd}, we know that $\alpha(g)$ and $\beta(g)$ agree on $\overline{\Int_{\delta_{L_x}}(L_x)}$ (respectively $\overline{\Int_{\delta_{R_x}}(R_x)}$). We are left to deal with the interval $C_x$. First, we expose and prove the following properties. \\
Since $\partial C_x\subseteq \overline{\Int_{\delta_{L_x}}(L_x)}\cup \overline{\Int_{\delta_{R_x}}(R_x)}$, we get that
\begin{equation}
\alpha(g)(t)=\beta(g)(t) \text{ for all } t\in \partial{C_{x}}.
\end{equation}
Combining \autoref{lma:glue} and the fact that $\alpha(g),\alpha({g_{1/2^n}})$ are constant on $\overline{C_x}\setminus\{x\}$, we also get 
\begin{equation}
\alpha(g)(t)\leq \alpha(g_{1/2^n})(t') \text{ for all } t,t'\in \overline{C_{x}}.
\end{equation}
We know that during the time interval $\overline{C_x}$, the path matrix $p^\lambda_{l,l'}\circ i_x$ is moving from $\diag(\lambda\mid \lambda\in X_l)$ to $\diag(\sigma_{l,l'}(\lambda)\mid \lambda\in X_l)$ in such a way that $p^\lambda_{l,l'}\circ i_x(t)\in \mathcal{B}(\lambda,\frac{2}{2^n})$ for all $t\in \overline{C_x}$ and all $\lambda\in X_l$. In other words, the unitary element $u$ is moving each point $\lambda$ of $X_l$ to a (unique) point $\sigma_{l,l'}(\lambda)$ of $X_{l'}$ staying at all time at distance at most $\frac{2}{2^n}$ from $\lambda$. A fortiori, we obtain that
\begin{equation}
\beta(g)(t)\leq \beta(g_{2/2^n})(t') \text{ for all } t,t'\in \overline{C_{x}}.
\end{equation}
Let $t\in C_x$ and let $t'\in \partial C_x$. Then
\[
\left\{
\begin{array}{ll}
\beta(g)(t)\leq \beta(g_{2/2^n})(t')=\alpha(g_{2/2^n})(t')\leq \alpha((g_{2/2^n})_{1/2^n})(t)=\alpha(g_{3/2^n})(t).\\
\alpha(g)(t)\leq \alpha(g_{1/2^n})(t')=\beta(g_{1/2^n})(t')\leq \beta((g_{1/2^n})_{2/2^n})(t)=\beta(g_{3/2^n})(t).
\end{array}
\right.
\] 
(We have successively applied (3)-(1)-(2) in the first inequality and (2)-(1)-(3) in the second.) Ultimately, we have shown that $\alpha(g)\leq \beta(g_{3/2^n})$ and $\beta(g)\leq \alpha(g_{3/2^n})$ for any $g\in \Lambda_n$. The result follows after observing that for any $h,h'\in \Lambda_{n-2}$ such that $h\ll h'$, then $h\ll h_{3/2^n}\ll h'$.

The general case now follows from the observation that any $\Cu$-morphism $\alpha\colon\Lsc(\T,\overline{\N})\longrightarrow \Cu(\bigoplus_1^n M_{d_k}(X_k))$ with $\alpha(1_\T)=[1_B]$ can be decomposed into a direct sum of $\Cu$-morphisms \break$\alpha_k\colon\Lsc(\T,\overline{\N})\longrightarrow \Cu(M_{d_k}(X_k))$ with $\alpha_k(1_\T)=[1_{M_{d_k}(X_k)}]$, that is, $\alpha=(\alpha_k)_1^n$.
\end{proof}

We would like to deduce from the previous theorem that we can find an actual lift. However, with current methods and a fail of the uniqueness result, we are only able to deduce a weaker result in terms of what we call \emph{Cauchy sequences}. These were originally introduced in the first version of this manuscript. In the meantime, the author has obtained a more general result in a joint work with E. Vilalta. We are stating an adapted version of the general statement for our particular context and we omit the proof. We refer the reader to \cite[Definition 3.4/3.5 - Theorem 3.8]{CV23} for more details.

\begin{lma}
\label{lma:Cauchy}
Let $T$ be a $\Cu$-semigroup and let $C\in R_+$ be a positive constant. Let $(\alpha_n)_n$ be a sequence in $\Hom_{\Cu}(\Lsc(\T,\overline{\N}),T)$ such that $dd_{\Cu}(\alpha_{n-1},\alpha_n)\leq \frac{C}{2^n}$ for any $n\in\N$. 
Then the sequence converges with respect to $dd_{\Cu}$ towards a unique $\Cu$-morphism $\alpha\colon \Lsc(\T,\overline{\N})\longrightarrow T$. 
\end{lma}

\begin{proof}
This is a particular case of \cite[Theorem 3.8]{CV23}.
\end{proof}

As in \cite{CV23}, we may refer to $(\alpha_n)_n$ as a \emph{Cauchy sequence} and to $\alpha$ as the \emph{limit of the sequence}, or simply say in this context, that the sequence $(\alpha_n)_n$ converges towards $\alpha$ with respect to $dd_{\Cu}$.

\begin{cor}
Let $B$ be a direct sum of matrix algebras of continuous functions over 1-dimensional compact $\CW$-complexes. Let $\alpha:\Lsc(\T,\overline{\N})\longrightarrow \Cu(B)$ be a $\Cu$-morphism such that $\alpha(1_\T)=[1_B]$.

Then there exists a sequence of unitary elements $(u_n)_n$ in $B$ such that $(\Cu(\varphi_{u_n}))_n$ is Cauchy and converges towards $\alpha$ with respect to $dd_{\Cu}$.
\end{cor}

\begin{proof}
Let $n\in\N$. By \autoref{thm:existencefx} we know that there exists a unitary element $u_n$ in $B$ such that $\beta_n:=\Cu(\varphi_{u_n})\underset{\Lambda_{n-2}}\simeq \alpha$.  We obtain a sequence $(u_n)_n$ of unitary elements in $B$ and for any $n,m\in\N$ with $n\leq m$, we compute that
\begin{align*}
dd_{\Cu}(\beta_n,\beta_m)&\leq 2 (dd_{\Cu}(\beta_n,\alpha)+dd_{\Cu}(\alpha,\beta_m))\\
&\leq 2(\frac{1}{2^{n-2}}+\frac{1}{2^{m-2}})\\
&\leq \frac{1}{2^{n-4}}.
\end{align*}

By \autoref{lma:Cauchy}, we deduce that $(\beta_n)_n$ converges towards a (unique) $\beta:\Lsc(\T,\overline{\N})\longrightarrow \Cu(B)$. Finally, observe that
\[
dd_{\Cu}(\alpha,\beta)\leq 2 (dd_{\Cu}(\alpha,\beta_n)+dd_{\Cu}(\beta_n,\beta))
\leq 2 dd_{\Cu}(\alpha,\beta_n)+ 2 dd_{\Cu}(\beta_n,\beta)
\]
from which we conclude that $dd_{\Cu}(\alpha,\beta)=0$, or equivalently, that $\alpha=\beta$.
\end{proof}

\begin{thm}
Let $A$ be a unital $\AH_1$-algebra. Let $\alpha:\Lsc(\T,\overline{\N})\longrightarrow \Cu(A)$ be a weakly semiprojective $\Cu$-morphism such that $\alpha(1_\T)=[1_A]$.

Then there exists a sequence of unitary elements $(u_n)_n$ in $A$ such that $(\Cu(\varphi_{u_n}))_n$ is Cauchy and converges towards $\alpha$ with respect to $dd_{\Cu}$.
\end{thm}

\begin{proof}
Arguing similarly as in the proof of \autoref{thm:classificationAF}, we can find a sequence $(u_n)_n$ of unitary elements of $A$ such that $dd_{\Cu}(\Cu(\varphi_{u_n}),\alpha)\leq \frac{C}{2^n}$ for some constant $C>0$. We deduce that $dd_{\Cu}(\Cu(\varphi_{u_{n-1}}),\Cu(\varphi_{u_n}))\leq \frac{C'}{2^n}$ for some constant $C'>0$ and the result follows from \autoref{lma:Cauchy}.
\end{proof}

\section{Real rank zero and beyond}
This final section explores the extension of uniqueness results to settings beyond the $\AF$-case, within the stable rank one context. A first and very natural place to start is the real rank zero case, as these $\CatCa$-algebras enjoy a lot of nice properties such as an abundance of projections or the $(1+\epsilon)$ exponential rank for unitary elements that are connected to the identity. In a first part, we show that $\Cu$ distinguishes unitary elements of $\CatCa$-algebras of stable rank one and real rank zero that have trivial $\K_1$-group. As a corollary, we are able to obtain a weaker, but somehow less satisfactory, result when removing the trivial $\K_1$-hypothesis. Finally, we exhibit an obstruction to uniqueness outside of the real rank zero setting by constructing two unitary elements in a $\CatCa$-algebra of stable rank and real rank one, namely the Jiang-Su algebra $\ZZ$, that agree at level of $\Cu$ but are not approximately unitarily equivalent. This obstruction reveals once more the importance of the de la Harpe-Skandalis determinant in the classification of unitary elements.
\subsection{Real rank zero - On the path towards Uniqueness}
Let us recall some nice properties that we have access to whenever being in the stable rank one and real rank zero setting. Let $A$ be a unital $\CatCa$-algebra of stable rank one. It is known that the functor $\Cu$ classifies self-adjoint elements of $A$. (See e.g. \cite{CE08,RS10}.)
If moreover $A$ has real rank zero, then it is shown in \cite{L14} that any self-adjoint can be approximated by a linear combination of projections up to arbitrary precision. Further, any unitary element in the connected component of the identity can be approximated by the exponential of a self adjoint element up to arbitrary precision.

\begin{prop}
Let $A$ be a unital $\CatCa$-algebra of stable rank one. Let $h,h'$ be self-adjoint elements in $A$. Then $d_{\Cu}(\Cu(\varphi_{e^{2i\pi h}}),\Cu(\varphi_{e^{2i\pi h'}}))\leq d_{\Cu}(\Cu(\varphi_{h}),\Cu(\varphi_{h'}))$. The converse inequality holds (and hence, there is equality) whenever $\spectrum(e^{2i\pi h})\cup \spectrum (e^{2i\pi h'})\neq\T$. 
\end{prop} 

\begin{proof}
The idea behind this proof relies on the fact that for any $h\in A_{sa}$, the element $u_h:=e^{2i\pi h}$ obtained by functional calculus belongs to $\CatCa(1,h)$ and that whenever $\spectrum (u_h)$ is not full, there exists a well-chosen (continuous) logarithm $\log:\CC_0(\spectrum (u_h))\longrightarrow \C$ such that $\log(u_h)=h$ and hence, $h$ belongs to $\CatCa(u_h)$. Under such circumstances, we have that $\CatCa(1,h)=\CatCa(u_h)$ and hence, $\varphi_h$ and  $\varphi_{u_h}$ determine one another. 

More specifically, let $h,h'\in A_{sa}$ be contractions and let $\exp\in\CC([-1,1])$ be the map that sends $t\longmapsto e^{2i\pi t}$. Now define $u_h:=\exp(h)$ and $u_{h'}:=\exp(h')$. Then, for any $g\in \CC(\T)$, we have $\varphi_{u_h}(g)= \varphi_h(g\circ \exp)$. (Similarly with $h'$.) Observe that for any open arcs $U,V$ in $\T$ such that $\overline{U}\subseteq V$ and any $r>0$ such that $U_r\subseteq V$, we have that $(\exp^{-1}(U))_r\subseteq \exp^{-1}(V)$ in $[-1,1]$. Also, for any $g\in \CC(\T)$ such that $\supp g = U$, we have $\supp(g_U\circ \exp)=\exp^{-1}(U)$. In particular, let $r:=d_{\Cu}(\Cu(\varphi_{h}),\Cu(\varphi_{h'}))$ let $U,V$ in $\T$ such that $U_r\subseteq V$, we have that $\Cu(\varphi_{u_h})(1_U)=\Cu(\varphi_{h})(1_{\exp^{-1}(U)})\leq \Cu(\varphi_{h})(1_{(\exp^{-1}(U))_r})\leq \Cu(\varphi_{h'})(1_{\exp^{-1}(V)})=\Cu(\varphi_{u_h'})(1_V)$. Similarly, interchanging $h$ and $h'$. It follows that $d_{\Cu}(\Cu(\varphi_{e^{2i\pi h}}),\Cu(\varphi_{e^{2i\pi h'}}))\leq d_{\Cu}(\Cu(\varphi_{h}),\Cu(\varphi_{h'}))$.

A similar argument gives us the converse inequality, but one needs to make sure that there exists a continuous logarithm that \textquoteleft simultaneously\textquoteright\ gives back $h$ and $h'$. It is sufficient that there exists a point $\lambda\in \T$ that does not belong to $\spectrum (u)\cup \spectrum (v)$.
\end{proof}

\begin{prop}
Let $A$ be a unital $\CatCa$-algebra with stable rank one. Let $u,v$ be unitary elements in $A$ such that $dd_{\Cu}(\Cu(\varphi_{u}),\Cu(\varphi_{v}))\leq \frac{1}{2^n}$ and such that $\spectrum (u) \cup \spectrum (v)\neq \T$. 

Then there exists a unitary element $w\in A$ such that $\Vert wuw^*-v\Vert\leq \frac{1}{2^{n-3}} $.
\end{prop}

\begin{proof}
As stated before, the fact that there exists $\lambda\in \T$ such that $\lambda$ does not belong to $\spectrum (u) \cup \spectrum (v)\neq \T$ allows us to find a \textquoteleft common\textquoteright\ continuous logarithm $\log\in C(\T)\setminus\{\lambda\}$. Applying functional calculus, we get two self adjoint contractions $h:=\log(u)$ and $l:=\log(v)$. Arguing similarly as in the previous proof, we know that $dd_{\Cu}(\Cu(\varphi_{h}),\Cu(\varphi_{l}))\leq \frac{1}{2^n}$. It follows that $d_{\Cu}(\Cu(\varphi_{h}),\Cu(\varphi_{l}))\leq \frac{1}{2^{n-1}}$. Now using Ciuperca and Elliott's result (see \cite{CE08}), improved by Robert and Santiago (see \cite[Theorem 1]{RS10}), we deduce that $d_U(h,l)\leq \frac{1}{2^{n-3}}$. That is, there exists a unitary element $w\in A$ such that $\Vert whw^*-l\Vert\leq \frac{1}{2^{n-3}}$. The result follows after noticing that for any $f\in \CC_0([-1,1]\setminus\{0\})$ and any unitary $w\in A$, we have that $f(whw^*)=wf(h)w^*$. (The latter claim is easily proven for polynomial functions first and then for any continuous function by a density argument.) 
\end{proof}

Observe that the assumption $\spectrum (u) \cup \spectrum (v)\neq \T$ (which implies that $[u]=[v]=0_{\K_1(A)})$ is the best we can achieve. If we fix $X=[0,1]$ in the example of \autoref{subsec:4A}, we obtain two unitary elements of a $\CatCa$-algebra with stable rank one that agree at level of $\Cu$, have the same (trivial) $\K_1$-class but are not approximately unitarily equivalent. (The union of the spectrums is indeed covering $\T$.) We next narrow down to the real rank zero and stable rank one setting and we are able to get rid of the assumption on the spectra.

\begin{thm}
Let $A$ be a unital $\CatCa$-algebra with stable rank one and real rank zero. Let $u,v$ be unitary elements in $A$ such that $dd_{\Cu}(\Cu(\varphi_{u}),\Cu(\varphi_{v}))\leq \frac{1}{2^n}$ and $[u]_{\K_1}=[v]_{\K_1}$. 

(i) For any unitary element $z$ in $A$ such that $[z]+[u]=0$ in $\K_1(A)$, e.g. $u^*$ or $v^*$, there exists a unitary element $w\in M_2(A)$ such that $\Vert w(u\oplus z)w^*-(v\oplus z)\Vert\leq \frac{1}{2^{n-7}} $. 

(ii) If moreover $u$ (and $v$) are in the connected component of the identity, then $w$ can be chosen in $A$. In other words, there exists a unitary element $w\in A$ such that $\Vert wuw^*-v\Vert\leq \frac{1}{2^{n-7}} $. 
\end{thm}

\begin{proof}
Let us first prove case (ii) where $[u]=[v]=0$ in $\K_1(A)$. We use the $(1+\epsilon)$ exponential rank obtained in \cite[Theorem 5]{L93} to find self-adjoint elements $h',l'$ in $A$ such that $\Vert u- e^{2i\pi h'} \Vert\leq \frac{1}{2^{n+2}}$ and $\Vert v- e^{2i\pi l'} \Vert\leq \frac{1}{2^{n+2}}$. Next, we use the characterization of real rank zero given in \cite{L14}, to find self-adjoints elements $h,l$ with finite spectrum and such that $\Vert h'- h\Vert\leq \frac{1}{2^{n+2}}$ and $\Vert l'- l\Vert\leq \frac{1}{2^{n+2}}$. We now compute that $\Vert u- e^{2i\pi h} \Vert\leq \frac{1}{2^{n+1}}$ and $\Vert v- e^{2i\pi l} \Vert\leq \frac{1}{2^{n+1}}$. Also that
\begin{align*}
dd_{\Cu}(\varphi_{e^{2i\pi h}},\varphi_{e^{2i\pi l}})&\leq 2dd_{\Cu}(\varphi_{e^{2i\pi h}},\varphi_{u})+4dd_{\Cu}(\varphi_{u},\varphi_{v})+4dd_{\Cu}(\varphi_{v},\varphi_{e^{2i\pi l}})\\
&\leq \frac{1}{2^{n}}+ 4\frac{1}{2^{n}}+ 2\frac{1}{2^{n}}\\
&\leq \frac{1}{2^{n-3}}
\end{align*}
We are now able to compute that
\begin{align*}
d_U(u,v)&\leq d_U(u,e^{2i\pi h})+d_U(e^{2i\pi h},e^{2i\pi l})+d_U(v,e^{2i\pi l})\\
&\leq \frac{1}{2^{n+1}}+8dd_{\Cu}(\varphi_{e^{2i\pi h}},\varphi_{e^{2i\pi l}})+\frac{1}{2^{n+1}}\\
&\leq \frac{1}{2^{n-7}}
\end{align*}
where we have applied the previous proposition to $e^{2i\pi h}$ and $e^{2i\pi l}$ in the second line.
The general case (i) follows after noticing that $M_2(A)$ also has real rank zero and that $[u\oplus z]=[v\oplus z]=0$ in $\K_1(A)$ and applying the latter result.
\end{proof}

\begin{cor}
\label{cor:RR0}
The functor $\Cu$ distinguishes unitary elements of any unital $\CatCa$-algebra with stable rank one, real rank zero and trivial $\K_1$-group.
\end{cor}

\subsection{Real rank one - Obstruction to Uniqueness}
As observed earlier, the real rank zero is a necessary condition in the latter corollary. Although this fact has already been illustrated in the example constructed in \autoref{subsec:4A}, by providing an obstruction in the real rank one setting, we further reinforce it by presenting an example in the well-known Jiang-Su algebra $\ZZ$. More specifically, we exhibit two unitary elements $u,v$ in $\ZZ$ such that $\Cu(\varphi_u)=\Cu(\varphi_v)$ and yet $u\nsim_{aue} v$. Apart from providing another obstruction to uniqueness in the real rank one setting, (and the simple case this time) the subsequent construction is also worth considering in its own right for the methods it employs.

Let us start by recalling some facts about the Jiang-Su algebra $\ZZ$. We refer the reader to \cite[Section 7]{GP23} for more details. The Jiang-Su algebra is a unital simple (infinite-dimensional) $\CatCa$-algebra that has stable rank one, real rank one, trivial $\K_1$-group and a unique tracial state $\tau:\ZZ\longrightarrow \C$. (And hence, there exists a unique trace that we also write $\tau:\ZZ_+\longrightarrow \R_+$).
The Cuntz semigroup of the Jiang-Su algebra $\Cu(\ZZ)$ is isomorphic to $ \N\,\sqcup\, ]0,\infty]$ where the mixed sum and mixed order are defined as follows: for any $x_c:=k \in \N$, any $x_s:=k\in\, ]0,\infty]$ and any $\epsilon>0$, we have that $x_s\leq x_c\ll x_c \leq x_s+\epsilon$. Moreover, $x_c+x_s=2x_s$. A fortiori, the compact elements are $\N$ and the non-compact elements $]0,\infty]$ constitute an absorbing $\Cu$-subsemigroup. Finally, we have a dichotomy that will be of use in the sequel. (See also \cite{BPT08}.) 

\emph{Fact}: Let $a\in (\ZZ\otimes \KK)_+$ and let $d_\tau:\Cu(\ZZ)\longrightarrow \overline{\R}_+$ be the functional induced by the (unique) trace $\tau$ on $\ZZ$. Then one (and only one) of the following holds. \\Either

(i) $\{0\}$ is an not isolated point in $\spectrum(a)\cup\{0\}$, which is equivalent to

(i') There exists $x_s\in ]0,\infty]$ with $[a]=x_s$, i.e. $[a]$ is not compact, and we have $d_\tau([a])=x_s$. \\
Or

(ii) $\{0\}$ is an isolated point in $\spectrum(a)\cup\{0\}$, which is equivalent to

(ii') There exists $x_c\in \N$ with $[a]=x_c$, i.e. $[a]$ is compact, and we have $d_\tau([a])=x_s$. (Where $x_s$ is the \textquoteleft non-compact version\textquoteright\ of the integer $x_c$.) \\

We next recall a theorem by Markov-Riesz that links measures on the Borel algebra $\mathcal{B}(X)$ and tracial states on the $\CatCa$-algebra $\CC_0(X)$. This will allow us to assign a measure to each positive element of $\ZZ$. 

\begin{thm}[Riesz Theorem for measures] 

Let $X$ be a Hausdorff locally compact space and let $\tau:\CC_0(X)\longrightarrow \C$ be a tracial state. Then there exists a unique extended Borel measure $\mu:\mathcal{B}(X)\longrightarrow \overline{\R}_+$ that is finite on compact sets and such that $\tau(f)=\int_X fd_{\mu}$.
\end{thm}

Now, let $a\in \ZZ_+$ be a positive element. Post-composing $\phi_a$ with $\tau$, we obtain a faithful tracial state $\tau_a:=\tau\circ \phi_a:\CC_0(\spectrum(a)\setminus\{0\})\longrightarrow \C$. Using the Markov-Riesz theorem, we can find a unique measure $\mu_a$ such that $\tau_a(f)=\int_{\spectrum(a)\setminus\{0\}}fd\mu_a$ for any $f\in \CC_0(\spectrum(a)\setminus\{0\})$. We will refer to this unique measure $\mu_a$ as the \emph{measure associted to $a$}. 

\begin{prop}
(i) For any $a\in \ZZ_+$ and any $f\in \CC_0(\spectrum(a)\setminus\{0\})_+$ we have that $d_\tau[f(a)]=\mu_a(\supp(f))$. In particular, $[f(a)]=\mu_a(\supp(f))$ whenever $[f(a)]$ is not compact.

(ii) For any $k\in \N$ there exists a positive element $a_k$ in $\ZZ$ with spectrum $]0,k]$ such that its associated measure $\mu_{a_k}$ is $\mathfrak{m}/k$, where $\mathfrak{m}$ is the Lebesgue measure.
\end{prop}

\begin{proof}
(i) Let $g_n:t\longmapsto t^{1/n}$. We know that $d_\tau([f(a)])=\lim\limits_n \tau(g_n\circ f(a))=\lim\limits_n \tau_a(g_n\circ f)=\lim\limits_n \int_{\spectrum(a)\setminus\{0\}} (g_n\circ f)d_{\mu_a}=\mu_a(\supp f)$. 

(ii) Let $k\in\N$ and consider the assignement
\[
\begin{array}{ll}
\alpha_k:\Lsc(]0,k],\overline{\N})\longrightarrow \Cu(\ZZ)\\ 
\hspace{2,1cm}f\longmapsto \int_{]0,k]}(f/k)d\mathfrak{m}
\end{array}
\]
It can be checked that $\alpha$ is a $\Cu$-morphism. (E.g. using the chain-decomposition properties of \cite[Theorem 4.4]{C22}.) By the deep result \cite[Theorem 1.0.1]{R12} of Robert, there exists a *-homomorphism $\varphi_k:\CC_0(]0,k])\longrightarrow \ZZ$ lifting $\alpha$, i.e. such that $\Cu(\varphi_k)=\alpha_k$. Finally, by the bijective correspondence between $\Hom_{\CatCa}(\CC_0(\R_+^*),\ZZ)$ and $\ZZ_+$, we obtain a positive element $a_k:=\varphi_k(\id)$ in $\ZZ$ whose spectrum is $]0,k]$ and associated measure is $\mathfrak{m}/k$.
\end{proof}

\begin{thm}
Let $k,l\in \N$ and consider the positive elements $a_k,a_l$ in $\ZZ$ as constructed in the above proposition. Write $u_k:=e^{2i\pi a_k}$ and $u_l:=e^{2i\pi a_l}$.

 Then $\Cu(\varphi_{u_k})=\Cu(\varphi_{u_l})$ and yet, $u\sim_{aue}v$ if and only if $k-l$ is even.
\end{thm}

\begin{proof}
Let $k,l\in \N$ and let $f\in \CC(\T)_+$. We first have to show that $\Cu(\varphi_{u_k})(f)=\Cu(\varphi_{u_l})(f)$. However, we know that $\Cu(\varphi_{u_k})(f)=\Cu(\varphi_{a_k})(f\circ \exp_k)$, where $\exp_k$ is the restriction of the exponential map $\exp:t\longmapsto e^{2i\pi t}$ to the domain $]0,k]$. Similarly for $u_l,a_l$ and $\exp_l$. Now using (i) of the above propostion, we get that 
\begin{align*}
\Cu(\varphi_{u_k})(f)&= \Cu(\varphi_{a_k})(f\circ \exp)\\
&=\mu_{a_k}(\supp (f\circ \exp_k))\\
&=\mathfrak{m}/k(\{t\in ]0,k]\mid f\circ \exp(t)\neq 0\})\\
&=\mathfrak{m}/l(\{t\in ]0,l]\mid f\circ \exp(t)\neq 0\})\\
&=\mu_{a_l}(\supp (f\circ \exp_l))\\
&=\Cu(\varphi_{u_l})(f).
\end{align*}

Next, let us compute the de la Harpe-Skandalis determinant of $u_k$. We observe that the codomain of (both stable and non-stable) de la Harpe-Skandalis determinant is $(\ZZ/\overline{[\ZZ,\ZZ]})/\overline{\tau(\K_0(\ZZ))}\simeq \C/\overline{\Z}\simeq \C/\Z$, where $\tau$ is the unique tracial state on $\ZZ$. Now, we compute that \[\overline{\Delta}_{\Tr}(e^{2i\pi a_k})=[\tau(a_k)]_{\C/\Z}=[\int_0^k td\mu_{a_k}]_{\C/\Z}=[\int_0^k (t/k)d\mathfrak{m}]_{\C/\Z}= [k/2]_{\C/\Z}.\]

The \textquoteleft only if\textquoteright\ part of the theorem now follows from \autoref{prop:dhs}, while the \textquoteleft if\textquoteright\ part of the theorem can be obtained via \cite[Theorem 6.6]{M11}, together with some extra clarifications. (Roughly, in this setting $KL(\varphi_{u_k})$ is entirely determined by $\K_*(\varphi_{u_k})$ and hence, $\varphi_{u_k}$ and $\varphi_{u_l}$ agree on $KL$. Furthermore, they agree on traces since they agree on their Cuntz semigroup. See e.g. \cite[Theorem 3.14]{C24}. Finally, from the above computation of the de la Harpe-Skandalis determinant, it is readily seen that the third condition of (2) in \cite[Theorem 6.6]{M11} is met whenever $k-l$ is even.)
\end{proof}

\section{Closing Remark}
In the aim of classifying a larger class of non-simple $\CatCa$-algebras than the class classified by Robert in \cite{R12}, it is clear that one needs a variation of the Cuntz semigroup that incorporates some kind of $\K_1$-information. The author have been introducing the unitary Cuntz semigroup $\Cu_1$ in \cite{C21a} and extensively studying its properties in \cite{C21a} and \cite{C21b}. As known results in this direction, we shall highlight that the $\K_*$-group, and a fortiori the unitary Cuntz semigroup, are complete invariants for $\A\!\T$-algebras of real rank zero. (See \cite[Theorem 7.3 - 7.4]{ELL93} and \cite[Theorem of the Erratum]{C23}.) Furthermore, the unitary Cuntz semigroup has been able to successfully distinguish two non-simple $\CatCa$-algebras  where the invariant $(\Cu,\K_1)$, and a fortiori the $\K_*$-group, could not. (See \cite{C23}.)

Nevertheless, as highlighted in the present manuscript via the obstruction examples, the incorporation of the Hausdorffized algebraic $\K_1$-group, written $\overline{\K}^{\alg}_1$ (which is intimately related to the de la Harpe-Skandalis determinant through the Nielsen-Thomsen sequence, see \cite{NT96}) instead of the usual $\K_1$-group seems indispensable if one wishes to extend the classification of Robert, say to $\A\!\T$-algebras. This is comforted by the fact that $\overline{\K}^{\alg}_1$ is used in the classification of simple $\A\!\T$-algebras done in \cite{NT96}. Consequently, the author has been working on a generic method, modeled after the construction of the unitary Cuntz semigroup, to build distinct variations of the Cuntz semigroup which enjoy numerous properties and tools for classification. (See \cite{C23b}.) A fortiori, this method allows to define an Hausdorffized version of the unitary Cuntz semigroup. This invariant is a strong candidate to the uniqueness part and hence the classification of unitary elements of $\AH_1$-algebras which would give a positive answer to \cite[Question 16.7]{GP23}.\\

\emph{\underline{Conjecture}:} The Hausdorffized unitary Cuntz semigroup classifies unitary elements of any unital $\AH_1$-algebra.\\

We point out that if the conjecture turns out to be true, we could obtain a complete classification of certain unital $\ASH_1$-algebras, including all $\A\!\T$-algebras as a corollary. (A torsion-free $\K_1$-group kind of assumption would surely arise without the incorporation of the total $\K$-Theory.) Broadly speaking, the approach would mimic the one developed in \cite{R12} by showing that the Hausdorffized unitary Cuntz semigroup satisfies similar permanency properties as the Cuntz semigroup (\cite[Theorem 3.2.2]{R12}) to conclude the classification of $\A\!\T$-algebras. The study of a reduction algorithm through a suitable equivalence relation, as done in \cite[5.1]{R12}, could even yield a larger classification within unital $\ASH_1$-algebras.

\end{document}